%% file: main.tex
\documentclass[10pt]{article} %size option 10 11 12, default 10
\usepackage[letterpaper]{geometry}
\input{macro}
\date{}
\hypersetup{allcolors = blue}
\pagecolor{white}
\color{black}

\title{Branched Covers of Hyperbolic Groups}
\author{Darius Alizadeh \\
\small{dariusdalizadeh@gmail.com}}

\begin{document}
\maketitle

\begin{abstract} 
Given a hyperbolic group $G$ and a quasiconvex subgroup $Q$, we define a \emph{branched cover of $G$ along $g$}, which is a hyperbolic group $H$ with a certain map into $G$. This builds on recent work on drilling hyperbolic groups and generalizes the case where $G$ is the fundamental group of a closed hyperbolic $3$--manifold $M$, $Q \cong \ZZ$ is represented by an embedded geodesic loop $\gamma$, and $H$ is the fundamental group of a branched cover of $M$ with branching locus $\gamma$. We show that certain deepness assumptions on Dehn fillings induce branched covers, providing many examples of such branched covers. Some additional assumptions imply these branched covers have boundary $S^2$, which may hold interest for the Cannon Conjecture. 

\end{abstract}

\input{Sections/1_Introduction}

\input{Sections/2_Background}

\input{Sections/3_BuildingBranchedCovers}

\bibliographystyle{alpha}
\bibliography{main.bbl}

\end{document}

%% file: macro.tex
\usepackage[english]{babel}
\usepackage{setspace}
\usepackage{amsmath}
\usepackage{amsthm}
\usepackage{amssymb}
\usepackage{amsfonts}
\usepackage{mathrsfs}
\usepackage{thmtools}
\usepackage{thm-restate}
\usepackage{graphicx}
\usepackage[colorlinks=true, allcolors=Nord8]{hyperref}
\usepackage{tikz-cd} 
\usepackage{quiver}
\usepackage{enumitem}
\usepackage{xcolor}
\usepackage{etoolbox}
\usepackage[colorinlistoftodos]{todonotes}
\usepackage{float}
\usepackage{theoremref}
\definecolor{Nord0}{HTML}{2e3440}

\definecolor{Nord4}{HTML}{d8dee9}

\definecolor{Nord7}{HTML}{8fbcbb}%light green
\definecolor{Nord8}{HTML}{88c0d0}%aqua
\definecolor{Nord9}{HTML}{81a1c1}%lighter blue
\definecolor{Nord10}{HTML}{5e81ac}%dark blue
\definecolor{Nord11}{HTML}{bf616a}%red
\definecolor{Nord12}{HTML}{d08770}%orange
\definecolor{Nord13}{HTML}{ebcb8b}%yellow
\definecolor{Nord14}{HTML}{a3be8c} %green
\definecolor{Nord15}{HTML}{b48ead} %pink

\pagecolor{Nord0} %Set the background color
\color{Nord4} %Set the defaut word color.

\theoremstyle{definition}
\newtheorem{theorem}{Theorem}[section]
\newtheorem{proposition}[theorem]{Proposition}
\newtheorem{corollary}[theorem]{Corollary}
\newtheorem{lemma}[theorem]{Lemma}

\newtheorem{definition}[theorem]{Definition}

\newtheorem{conjecture}[theorem]{Conjecture}

\newtheorem{notation}[theorem]{Notation}
\newtheorem{assumption}[theorem]{Assumption}

\newtheorem{remark}[theorem]{Remark}
\newtheorem{claim}{Claim}[theorem]
\setuptodonotes{inline, color=Nord13} %Problems that I can edit right now
\newcommand{\sys}[0]{\textrm{sys}}

\newcommand{\sS}[0]{\mathscr{S}}

\newcommand{\ZZ}[0]{\mathbb{Z}}

\newcommand{\llangle}{\langle\negthinspace\langle}
\newcommand{\rrangle}{\rangle\negthinspace\rangle}

\newcommand{\sH}[0]{\mathscr{H}}

\usepackage{tikz}
\usetikzlibrary{math}
\usepgfmodule{oo}
\usepackage{scalerel}
\usepackage{stackengine,wasysym}
\usepackage{mathtools}

\DeclareFontFamily{U}{mathx}{\hyphenchar\font45}
\DeclareFontShape{U}{mathx}{m}{n}{
      <5> <6> <7> <8> <9> <10>
      <10.95> <12> <14.4> <17.28> <20.74> <24.88>
      mathx10
      }{}
\DeclareSymbolFont{mathx}{U}{mathx}{m}{n}
\DeclareFontSubstitution{U}{mathx}{m}{n}
\DeclareMathAccent{\widecheck}{0}{mathx}{"71}

\newcommand{\leftQ}[2]{\left.\raisebox{-.2em}{$#2$}\middle\backslash\raisebox{.2em}{$#1$}\right.}

\usepackage{thm-restate}

%% file: Sections/1_Introduction.tex
\section{Introduction}

A fundamental strategy for studying Gromov hyperbolic groups is to develop analogs of behavior seen in hyperbolic manifolds. Examples include the boundary at infinity, Dehn fillings, exponential growth, quasi-geodesic stability, the geodesic flow, and much more. This paper is about applying this strategy to branched covers of hyperbolic $3$--manifolds with branching locus an embedded geodesic loop. We recall the general definition of a branched cover.

\begin{definition}
    Suppose $b:Y \longrightarrow X$ is some continuous function and $L\subset X$. Then $b$ is a \emph{branched cover of $X$ with branching locus $L$} if $b$ is a covering map when restricted to $Y \setminus b^{-1}(L)$.  The \emph{degree} of $b$ is the degree of the restriction of $b$. 
\end{definition}

The most basic examples are maps $p_n:\mathbb{C} \longrightarrow \mathbb{C}$ given by $z\longrightarrow z^n$ with $L = \{0\}$. This leads to the well studied context where $X,Y$ are Riemann surfaces, $L$ is a finite collection of points, and for each $x \in L$, the branched cover is locally modeled on the restriction of $p_n$ to the unit disc for some $n$, with $x$ playing the role of the origin. Similarly, we can construct branched covers of the solid torus $\pi_n:D^2\times S^1 \longrightarrow D^2\times S^1$ given by $\pi_n(z,\theta) = (z^n,\theta)$ with $L$ as the core of the solid torus. Variations on this theme, like using a cover of the circle by itself in the second coordinate, or by cutting the solid torus along a meridian and gluing it back together with a rotation, give even more examples. In fact, every finite cover of the torus by itself can be extended to a branched cover of solid tori over the core using these variations. 

The focus of this paper is generalizing the case where $Y,X$ are hyperbolic $3$--manifolds, $b$ has finite degree, and $L=\gamma$ is an embedded geodesic loop in $X$. In \cite{Drilling}, the authors develop the following group theoretic analog of drilling such an $X$ along $\gamma$:

\begin{definition}(Drilling)\thlabel{defn drilling}
    Suppose $G$ is a hyperbolic group and $g\in G$ generates an infinite maximal cyclic subgroup. A \emph{drilling of $G$ along $g$} is a relatively hyperbolic pair $(\widehat{G},P)$ along with a subgroup $N\triangleleft P$ and an identification $P/N \cong \langle g \rangle$ so that $\widehat{G}/\llangle N \rrangle \cong G$ and the quotient identifies $P/N$ with $\langle g\rangle$.
\end{definition}

Here the maximal cyclic subgroup plays the role of $\gamma$, as we explain below. The arguments in this paper work in slightly more generality, leading us to the following definition. Recall that if $(\widehat{G},P)$ is a relatively hyperbolic pair and $N \triangleleft P$, then the \emph{Dehn filling of $\widehat{G}$ by $N$} is the quotient $\widehat{G}/\llangle N \rrangle$. Any finite index subgroup $\widehat{H}$ of such a $\widehat{G}$ inherits a relatively hyperbolic structure and an induced Dehn filling from $P$ and $N$, see \thref{induced filling} for details.

\begin{definition}\thlabel{defn generalized}
    Let $G$ be a hyperbolic group with $Q$ a quasiconvex subgroup. A \emph{generalized drilling of $G$ along $Q$} is a relatively hyperbolic pair $(\widehat{G},P)$ along with a subgroup $N \triangleleft P$ and an identification of $P/N \cong Q$ so that $\widehat{G}/\llangle N \rrangle \cong G$ and the quotient identifies $P/N$ with $Q$. A \emph{generalized branched cover of $Q$} is a finite index subgroup $\widehat{H}$ of $\widehat{G}$ so that the induced filling of $\widehat{H}$ by $N$ is hyperbolic. If $Q$ is a maximal cyclic subgroup of $G$, then we simply call this induced filling a \emph{branched cover}.
\end{definition}

Our main theorem is that, given a generalized drilling, we can produce many generalized branched covers. More explicitly, we show that deep enough normal subgroups always produce branched covers from any induced filling and that deep enough fillings induce branched covers for any finite index subgroup. Recall that a $PD(3)$ pair is a group pair whose cohomology satisfies a certain duality relationship, and this duality has implications for the boundary of a (relatively) hyperbolic group, see Section \ref{PDn stuff} for details. 

\begin{restatable}{theorem}{MainTheorem}(Main Theorem)\label{Main Theorem}
    Suppose $(\widehat{G},P,N)$ is a generalized drilling of $G$ along $Q$. Let $\widehat{H}$ be a finite index subgroup of $\widehat{G}$ with $H$ the induced filling of $\widehat{H}$ by $N$. There exists a finite set $F\subset \widehat{G}$ satisfying the following.
    \begin{enumerate}
        \item If $\widehat{H}$ is normal in $\widehat{G}$ and $F \cap \widehat{H} = \varnothing$, then $H$ is a branched cover of $G$.
        \item If $ F \cap N = \varnothing$, then $H$ is a branched cover of $G$.
    \end{enumerate}
    In each case, if $(\widehat{G},P)$ is a $PD(3)$ pair and $N \cong Q \cong \ZZ$, then $H$ is a $PD(3)$ group and $\partial H \cong S^2$. 
\end{restatable}

The assumption of normality in the first part is not ideal but it is not unreasonable; each finite index subgroup contains a finite index normal subgroup, namely the normal core, and the assumption of normality has interesting geometric implications which we explain below. 

Every relatively hyperbolic group $(\widehat{G},P)$ has an action on a pleasant hyperbolic space, namely the cusped space developed by Groves and Manning in \cite{DehnFillinginRHG}, which is constructed by attaching combinatorial horoballs to cosets of peripheral subgroups in a Cayley graph for $\widehat{G}$. For long enough fillings of such a $\widehat{G}$, the quotient of the cusped space by the filling kernel has large injectivity radius near the Cayley graph. As a technical tool, we generalize to \emph{pseudo--cusped spaces} (see \thref{defn pseudo cusped space}) and prove \thref{main metric lemma try2} which gives metric conditions guaranteeing that the filling of a relatively hyperbolic group is relatively hyperbolic. With the right maneuvering, a cusped space yields a suitable pseudo--cusped space, and combining the conclusion of \thref{main metric lemma try2} with implications from the definition of a generalized drilling gives Theorem \ref{Main Theorem}. The following summarizes the technical requirements that go into the choice of $F$, rephrasing Theorem \ref{Main Theorem} in an effective form. We hope this makes our work easier to take off the shelf. 

\begin{restatable}{theorem}{effectivemain}\label{effective main theorem}
    Suppose $(\widehat{G},P,N)$ is a generalized drilling of $G$ along $Q $ and $(\widehat{Y},\widecheck{Y})$ is a pseudo--cusped space for $(\widehat{G},P)$. Let $\widehat{H}$ be a finite index subgroup of $\widehat{G}$, let $\mathbb{P}_{\widehat{H}}$ be the induced peripheral structure for $\widehat{H}$, let $K \triangleleft \widehat{G}$ and $K_{\widehat{H}} \triangleleft \widehat{H}$ be the corresponding filling kernels, and let $H = \widehat{H}/K_{\widehat{H}}$ be the induced filling of $\widehat{H}$ by $N$. 
    
    \begin{enumerate}
        \item If the action of $(\widehat{H},\mathbb{P}_{\widehat{H}})$ and $K_{\widehat{H}}$ on $(\widehat{Y},\widecheck{Y})$ satisfies the assumptions of \thref{main metric lemma try2}, then $H$ is a branched cover of $G$.
        
        \item If the action of $(\widehat{G},P)$ and $K$ on $(\widehat{Y},\widecheck{Y})$ satisfies the assumptions of \thref{main metric lemma try2}, then $H$ is a branched cover of $G$.

    \end{enumerate}
    Suppose further that $(\widehat{G},P)$ is a $PD(3)$ pair and $N \cong Q\cong  \ZZ$. In either case, if the considered action on $(\widehat{Y},\widecheck{Y})$ also satisfies the very translating condition, then $H$ is a $PD(3)$ group and $\partial H \cong S^2$.
\end{restatable}

In \cite{Drilling}, the authors show that if $G$ is a hyperbolic group with $\partial G\cong S^2$, $g \in G$ generates a maximal cyclic subgroup, and $G$ acts geometrically on a hyperbolic space $X$ with the axes of conjugates of $g$ being sufficiently separated, then they can produce a drilling $(\widehat{G},P,N)$ of $G$ along $g$ with $\partial _P\widehat{G} \cong S^2$. This $\widehat{G}$ has a cusp uniform action on a hyperbolic space $\widehat{Y}$ with many properties, and we show these properties produce branched covers as well:
\begin{restatable}{theorem}{drillingsAreDeepEnough}\label{drillings are deep enough}
    Let $(\widehat{G},P,N)$ be a drilling of $G$ along $g$ constructed from the proof of\\ \cite[Theorem B]{Drilling}. Then every finite index subgroup of $\widehat{G}$ induces a branched cover of $G$. If, in addition, $G$ is torsion free, then each of these branched covers has boundary homeomorphic to $S^2$.
\end{restatable}

Let us explain in more detail why \thref{defn generalized} imitates the topological situation we are interested in. Suppose $X$ is a hyperbolic $3$--manifold and $b:Y \longrightarrow X$ is a branched cover over an embedded geodesic loop $\gamma$. In this case, $X^\circ = X \setminus \{\gamma\}$ is called a \emph{drilling of $X$ along $\gamma$} and we have this commutative diagram, where $\pi$ is the restriction of $b$ and $Y^\circ = Y \setminus b^{-1}(\gamma)$.
\begin{center}
\begin{tikzcd}
Y^\circ \arrow[r, "\pi"] \arrow[d] & X^\circ \arrow[d] \\
Y \arrow[r, "b"]                   & X                
\end{tikzcd}
\end{center}
It follows from Thurston's Geometrization Theorem that $X^\circ$ has a complete finite volume hyperbolic metric (see \cite{knotsInhyp3}), hence $\widehat{G}:= \pi_1(X^\circ)$ is hyperbolic relative to a subgroup $P \cong \mathbb{Z}^2$ corresponding to a torus $T$ which is the boundary of a tubular neighborhood of $\gamma$. Using Seifert--Van Kampen with $T$ and the fact that $T$ bounds a solid torus in $X$ with $\gamma$ as the core, we calculate that $\pi_1(X) = \widehat{G} /\llangle N \rrangle$, where $N \cong \mathbb{Z}$ is the subgroup of $P$ corresponding to a curve in $T$ which bounds a disc in $X$, punctured exactly once by $\gamma$. In other words, $\pi_1(X)$ is a group theoretic Dehn filling of $(\widehat{G},P)$ by $N$. The definition of a drilling is a collection of groups which play the same algebraic and geometric roles as the fundamental groups $X,X^\circ$. Maybe more explicitly, the definition of drilling is exactly what you get if you apply the functor $\pi_1$ to the inclusion $X^\circ \longrightarrow X$.

Because $Y^\circ$ is a finite cover of $X^\circ$, $\widehat{H} = \pi_1(Y^\circ)$ is a finite index subgroup of $\widehat{G}$. It follows that $\pi^{-1}(T)$ has finitely many components, say $T_1,\ldots T_n$. Each $T_i$ is a finite cover of $T$ and is a torus bounding a tubular neighborhood of some component of $b^{-1}(\gamma)$. Again we can use Seifert--Van Kampen to calculate that $\pi_1(Y)=\pi_1(Y^\circ)/\llangle N_i\rrangle$, where $N_i$ is the subgroup of $\pi_1(T_i)$ corresponding to a loop in $T_i$ which bounds a disc in $Y$ punctured exactly once by some component of $b^{-1}(\gamma)$. Any Dehn filling of a relatively hyperbolic group induces a filling for each finite index subgroup, and $H = \pi_1(Y)$ is exactly the induced filling of $\widehat{H}$ by $N$, viewed as a subgroup of $\widehat{G}$. Thus \thref{defn generalized} is exactly what you get if you apply the functor $\pi_1$ to the entire commutative square above. The hyperbolicity of $Y$ is not obvious from the geometry, but our definition requires that $H$ be a hyperbolic group.

Note that $b$ need not be $\pi_1$ surjective or injective. However, if $\widehat{H}$ is normal in $\widehat{G}$, then $\pi$ is a \emph{regular} cover and the deck group acts transitively on fibers. Thus the $T_i$ are all isomorphic as a covers of $T$ and we can use this to put an \emph{orbifold structure} on $X$. Orbifolds are spaces which are locally modeled on $\mathbb{R}^n$, except at points in the \emph{orbifold locus}, which are locally modeled on quotients of $\mathbb{R}^n$ by finite group actions. Here the orbifold locus is $\gamma$, which is locally homeomorphic to the core of the solid tori $T_i$, and the finite group acting is the set of deck transformations of $\pi$ which preserve a given $T_i$ -- these are $N/ (\widehat{H} \cap N) \cong \ZZ/ m \ZZ$ for some $m$. One can define \emph{orbifold fundamental groups} which track these finite groups, and all the tools of covering space theory can be developed for orbifolds, see for example \cite{orbifoldintro}. In particular, the orbifold fundamental group of a manifold is the standard fundamental group, an orbifold covering map induces an injection on orbifold fundamental groups, and orbifold fundamental groups always surject standard fundamental groups. This allows us to factor $b_*$ as an injection followed by a surjection:
\[ \pi_1(Y) = \pi_1^{orb}(Y) \hookrightarrow \pi_1^{orb}(X) \twoheadrightarrow \pi_1(X).\]

In this situation, $\pi_1^{orb}(X) = \widehat{G}/\llangle \widehat{H} \cap N \rrangle$, the image of $N$ in $\pi_1^{orb}(X)$ is isomorphic to $\ZZ/n\ZZ$, which represents the finite group acting to give the orbifold locus, and the surjection $\pi_1^{orb}(X) \longrightarrow \pi_1(X)$ is the quotient by the normal closure of the image of $N$. Therefore we could reasonably extend \thref{defn generalized} in the following way.

\begin{definition}
    Suppose $(\widehat{G},P,N)$ is a drilling of $G$ along $g$, and $\widehat{H}$ is finite index normal subgroup of $\widehat{G}$. Then $G^{orb} = \widehat{G}/\llangle \widehat{H} \cap N\rrangle$ is the \emph{orbifold fundamental group of $G$ induced by $\widehat{H}$} with \emph{orbifold locus} $N /\widehat{H} \cap N$.
\end{definition}

We do not explore these ideas in the present work, but historically, orbifolds have been a natural and useful geometric tool. One example is the Orbifold Theorem \cite{OrbThm}, which geometrizes $3$--orbifolds with nontrivial orbifold locus. Another example is the Smith Conjecture, which is in fact a theorem \cite{SmithConj}. It states that if $M$ is a homotopy $3$--sphere and $h$ is an orientation preserving periodic diffeomorphism of $M$ with fixed points, then the fixed point set $F(h)$ is an unknotted circle. It is not hard to see that $F(h)$ must be an embedded circle, so the quotient of $M$ by $h$ is an orbifold with orbifold locus a circle, and various orbifold theoretic tools connect $M$ and $M/\langle h \rangle$. The present work represents progress in porting these tools into a more general framework so that we can apply them more broadly. For example, one of the central open questions in geometric group theory is the Cannon Conjecture, which states that a hyperbolic group with boundary $S^2$ is virtually Kleinian. There are many variations on this theme, like the following:
\begin{conjecture}(Toral Relative Cannon Conjecture)
    Let $(G,P)$ be a relatively hyperbolic group pair with $P \cong \ZZ^2$ and Bowditch boundary homeomorphic to $S^2$. Then $G$ is virtually Kleinian. 
\end{conjecture}
In \cite[Theorem A]{Drilling}, the authors show that, in the residually finite case, the Toral Relative Cannon Conjecture implies the Cannon Conjecture. Such a group pair $(G,P)$ as in the Toral Relative Cannon Conjecture is \emph{Haken} in the sense of \cite[p. 282]{CoarseAlexander}, which generalizes the Haken property of $3$--manifolds. The Toral Relative Cannon Conjecture is thought to be easier because the Haken property offers the possibility of a hierarchy which enables inductive arguments to progress. In analogy with the $3$--orbifold setting for manifolds, we can hope that this orbifold setting for groups might be easier than the general case.

In \thref{defn generalized}, allowing $Q$ to be any quasiconvex subgroup gives some flexibility in producing examples. In the world of manifolds, this would correspond to drilling a hyperbolic $n$--manifold $M$ along a totally geodesic submanifold $S$ where the inclusion of $S$ is $\pi_1$--injective. Matthew Stover pointed out the following in a personal conversation: If $S$ is codimension 2, then a regular neighborhood $S$ in $M$ is homeomorphic to $S \times D^2$, hence its boundary is homeomorphic to $S \times S^1$. If $M^\circ = M\setminus S$ and $\alpha$ is a loop going once around the $S^1$ coordinate of this boundary, then Seifert--Van Kampen shows that $\pi_1(M) = \pi_1(M^\circ)/\llangle \alpha\rrangle$. Since $\langle \alpha \rangle \cong \mathbb{Z}$ is playing the role of the filling kernel, this gives geometric meaning to the condition that $N \cong \mathbb{Z}$ in Theorem \ref{Main Theorem}, and indicates that it should persist in higher dimensional analogs.

\begin{remark}
    In \cite{randombranched}, Cho, Lafont, and Skipper define and explore random branched covers of $2$--complexes. These are maps between two complexes which are coverings away from finitely many branch points in the interior of $2$--cells. Our focus is on the $3$--manifold case instead.
\end{remark}

We outline the sections of this paper. Section \ref{background} gives the necessary background in relatively hyperbolic groups, Dehn fillings, $PD(3)$ groups, etc. In Section \ref{sec 3} we define a pseudo--cusped space and use it to prove our metric criterion \thref{main metric lemma try2} for when a Dehn filling of a relatively hyperbolic group is relatively hyperbolic. Section \ref{sec building} has two subsections. In the first subsection, we prove the effective version of our main theorem, Theorem \ref{effective main theorem}, then prove Theorem \ref{Main Theorem} by building a suitable pseudo--cusped space for any relatively hyperbolic group and applying the effective version, Theorem \ref{effective main theorem}. The second subsection proves Theorem \ref{drillings are deep enough} about drillings produced in \cite{Drilling}.

Finally, we would like to thank Daniel Groves for suggesting and guiding us on this project.

%% file: Sections/2_Background.tex
\section{Background}\label{background}
\subsection{Combinatorial Horoballs and Relative Hyperbolicity}

\begin{definition}\cite[III.H Definition 1.1]{BH}
    Let $(X,d)$ be a geodesic metric space. A geodesic triangle in $X$ is $\delta$--\emph{slim} if each side is contained in the $\delta$--neighborhood of the other $2$ sides. We say $X$ is $\delta$--hyperbolic if every triangle is $\delta$--slim.
\end{definition}

There are many equivalent formulations of $\delta$--hyperbolicity, see \cite[III.H]{BH} for much more detail. One key property is that being $\delta$--hyperbolic for some $\delta$ is a quasi--isometry invariant \cite[III.H Theorem 1.9]{BH}. 

\begin{definition}[Combinatorial Horoball]\cite[Definition 3.1]{DehnFillinginRHG}
    Let $\Gamma$ be a connected simplicial graph. The \emph{combinatorial horoball based on $\Gamma$} is the graph $\sH(\Gamma)$ with vertex set $\Gamma \times \ZZ_{\geq 0}$ and two types of edges:
    \begin{enumerate}
        \item vertical edges connecting $(v,n)$ to $(v,n+1)$,
        \item horizontal edges connecting $(v,n)$ to $(w,n)$ whenever $d_\Gamma(v,w) \leq 2^n$. 
    \end{enumerate}
    We say that a vertex $(v,n)$ is at \emph{depth $n$} and extend this depth function linearly across edges. 
\end{definition}

The combinatorial horoball on a graph includes a natural copy of the graph at depth $0$, but contracts the metric exponentially as we go deeper. The result is always a hyperbolic space and we can understand how this warps the metric in the base graph:

\begin{lemma}\thlabel{horoballs are hyperbolic}\cite[Theorem 3.8]{DehnFillinginRHG}\cite[Lemma 4.9]{Drilling}
    Let $\Gamma$ be a simplicial graph and let $\sH = \sH(\Gamma)$ be the combinatorial horoball based on $\Gamma$. Then $\theta_0$--hyperbolic, where $\theta_0$ is a universal constant independent of $\Gamma$, and for any vertices $v,w \in \Gamma$, we have \[\frac{1}{2}d_\sH((v,0),(w,0)) -2 \leq \log_2(d_\Gamma(v,w)) \leq \frac{1}{2}d_{\sH}((v,0),(w,0)).\]
\end{lemma}

\begin{definition}\thlabel{defn cusped space}
    Let $G$ be a finitely generated group with $\mathbb{P}$ a finite collection of finitely generated subgroups, and choose a finite generating set $S$ for $G$ so that each $P\in \mathbb{P}$ is generated by $P \cap S$. Let $\Gamma$ be a Cayley graph for $G$ using $S$. The \emph{cusped space} $X(G,\mathbb{P})$ is built by attaching, for each $g\in G$ and $P \in \mathbb{P}$, a combinatorial horoball to $gP$ with the metric on $gP$ induced from $\Gamma$.
\end{definition}

The cusped space is not quite determined by $(G,\mathbb{P})$ because there is ambiguity in the choice of generating set for the Cayley graph. Any two choices are quasi--isometric by \cite[Corollary 6.7]{groff}, which makes the following definition well defined.

\begin{definition}\thlabel{defn rh}
    Let $G$ be a group with $\mathbb{P}$ a finite collection of subgroups. $G$ is \emph{hyperbolic relative to} $\mathbb{P}$ if any cusped space $X(G,\mathbb{P})$ is $\delta$--hyperbolic for some $\delta$. The \emph{Bowditch boundary} is $\partial_{\mathbb{P}}G := \partial X(G,\mathbb{P})$. 
\end{definition}

There are many equivalent definitions of relative hyperbolicity. Gromov originally defined relatively hyperbolic groups as groups with a \emph{cusp uniform action}, a definition which we recall now.

\begin{definition}\thlabel{horofunction and coarse horoball}
    Suppose $X$ is a $\delta$--hyperbolic space and $p \in \partial X$. For $x \in X$ and $\gamma$ a geodesic ray tending to $p$, define
    \[h_\gamma(x) = \limsup_{t\longrightarrow \infty}\big(d_\gamma(x,\gamma(t)) - t\big)\]
    The function $h_\gamma$ is a \emph{horofunction based at $p$}. A \emph{coarse horoball centered at $p$} is any set $\sH \subset X$ so that there are real numbers $a <b$ satisfying
    \[h^{-1}((-\infty,a]) \subset \sH \subset h^{-1}((-\infty, b]).\]
\end{definition}

\begin{definition}[Cusp Uniform Action]\thlabel{gromovs defn}\cite[Definition 8.6.A]{OriginalGromov}\cite[Definition 3.3]{Hruska_2010}
    Suppose $G$ acts on a proper $\delta$--hyperbolic metric space $X$ and $\mathbb{P}$ is a set of representatives for the conjugacy classes of maximal parabolic subgroups. Suppose also that there is a disjoint $G$--equivariant collection of coarse horoballs centered at the parabolic points of the action of $G$ on $\partial X$ with union $U$ open in $X$, and such that the quotient of $X \setminus U$ is compact. Then the action of $(G,\mathbb{P})$ on $X$ is called \emph{cusp uniform} and the space $X\setminus U$ is called the \emph{truncated space}.  
\end{definition}

The cusped space together with the deep parts of combinatorial horoballs is a particularly nice example of a cusp uniform action, and having a cusp uniform action is equivalent to being relatively hyperbolic:

\begin{theorem}\thlabel{equivalence of defns}\cite[Theorem 3.25]{DehnFillinginRHG}
    A group pair $(G,\mathbb{P})$ is relatively hyperbolic in the sense of \thref{defn rh} if and only if $(G,\mathbb{P})$ has a cusp uniform action.
\end{theorem}

See \cite{Hruska_2010} for many more definitions of relative hyperbolicity and their equivalence. The next statement helps us make the connection between the definition of relative hyperbolicity through cusped spaces and cusp uniform actions more explicit.

\begin{proposition}\thlabel{finally write it down}
    Suppose $\widecheck{Y}$ is a simplicial graph and $\widehat{Y}$ is constructed by adding combinatorial horoballs to a collection of connected subgraphs of $\widecheck{Y}$ with the induced path metric so that the resulting space $\widehat{Y}$ is $\delta$--hyperbolic for some $\delta$. If $\sH$ is one of these combinatorial horoballs and $\sH^L$ is the points of $\sH$ with depth at least $L$, then $\sH^L$ is a coarse horoball in the sense of \thref{horofunction and coarse horoball}.
\end{proposition}
\begin{proof}
    Fix some vertex $p \in \sH$ with depth $0$, let $\gamma$ be the vertical geodesic going straight down in $\sH$, and let $h_\gamma$ be the horofunction based at the limit point of $\sH$ using $\gamma$. 

    Suppose $x \notin \sH$. For any $t$, a geodesic from $x$ to $\gamma(t)$ must first travel to $\sH$, and then down to depth $t$, so 
    \[d(\gamma(t),x) -t \geq d(\sH,x) + t - t= d(\sH,x)\]
    and $h_\gamma(x) \geq d(\sH,x)$. Distances are always positive, so this implies $h^{-1}_\gamma((-\infty,\alpha]) \subset \sH$ for any $\alpha \leq 0$.

    Now suppose $x \in \sH$ is a vertex with depth $\alpha$. For any $t$, any path from $x$ to $\gamma(t)$ must travel at least $|t - \alpha|$ vertically to account for the change in depth. For $t \geq \alpha$, this means 
    \[d(\gamma(t),x) - t \geq t- \alpha - t = -\alpha\]
    and $h_\gamma(x) \geq - \alpha$. Therefore if $x \in h_{\gamma}^{-1}((-\infty,a])$ for some $a \leq 0$, we must have $-\alpha \leq a$. In particular, $h_{\gamma}^{-1}((-\infty,-L]) \subset  \sH^L$. 
    
    For the other inclusion, suppose $x \in \sH^L$ has depth $\alpha \geq L$, and let $\gamma$ be the geodesic starting at $x$ and traveling straight vertically down. Because distances shrink exponentially as we go deeper in $\sH$, there is some $t_0$ so that for all $t > t_0$, $d(\gamma(t+\alpha), \gamma_x(t)) = 1$, hence $d(\gamma(t + \alpha),x) \leq t + 1$. Using that $\alpha \geq L$, we have
    \[d(\gamma(t+\alpha),x) -(t + \alpha) \leq (t+ 1) - (t-\alpha) \leq 1-L,\]
    which implies $h_\gamma(x) \leq 1-L$ and $\sH^L \subset h^{-1}_\gamma((-\infty,1-L])$. 
\end{proof}

We are interested in finite index subgroups of relatively hyperbolic groups, which inherit a relatively hyperbolic structure from the larger group in the following way.

\begin{definition}\thlabel{induced peripheral}(Induced Peripheral Structure)
    Let $(G,\mathbb{P})$ be a relatively hyperbolic group, and suppose $H<G$ is a finite index subgroup. Let $\mathbb{P}_H$ be a finite collection of $H$--conjugacy class representatives for subgroups in the set  
    \[\{H \cap P^g \; | \; g \in G,\, P \in  \mathbb{P},\, |H\cap P^g| = \infty\}.\]
    Then $\mathbb{P}_H$ is the \emph{induced peripheral structure for $H$}.
\end{definition}

\begin{lemma}\thlabel{fin index subgroup is rel hyp}
    If $(G,\mathbb{P})$ is relatively hyperbolic and $H<G$ is a finite index subgroup with induced peripheral structure $\mathbb{P}_H$, then $(H,\mathbb{P}_H)$ is relatively hyperbolic and $\partial_{\mathbb{P}_H}H \cong \partial_{\mathbb{P}}G$.
\end{lemma}
\begin{proof}
    Suppose $(G,\mathbb{P})$ has a cusp uniform action on $X$. Then the conjugates of elements of $\mathbb{P}_H$ are exactly the maximal parabolic subgroups for $H$ acting on $X$, and the action of $H$ on the truncated space of $X$ is cocompact because $H$ is finite index in $G$ and the action of $G$ is cocompact. Thus any cusp uniform action for $(G,\mathbb{P})$ induces a cusp uniform action for $(H,\mathbb{P}_H)$, and the result follows from \thref{equivalence of defns}.
\end{proof}

There are many ways to define quasiconvexity for a subgroup of a (relatively) hyperbolic group, see \cite[\S 6]{Hruska_2010} for many of them. One straightforward definition uses the following notion of quasiconvexity in a general metric space.
\begin{definition}
    If $A$ is a subset of a geodesic metric space $(X,d)$, then $A$ is $K$--\emph{quasiconvex} if for any points $a,b \in A$ and any geodesic $\gamma$ between $a$ and $b$, $\gamma$ is contained in the $K$--neighborhood of $A$. 
    
\end{definition}
\begin{definition}
    
    Suppose $G$ is a hyperbolic group, and $X$ is a Cayley graph for $G$. A subgroup $Q$ of $G$ is \emph{quasiconvex} if it is $K$--quasiconvex for some $K$ as a subset of $X$. 
\end{definition}

It is important to note that the notion of quasiconvexity for hyperbolic groups is independent of the choice of Cayley graph, but \emph{does} depend on the Cayley graph for other groups, for example $\ZZ^2$. The only way we use quasiconvexity in this work is this theorem.

\begin{theorem}\cite[III.$\Gamma$.3.6]{BH}
    If $Q$ is a quasiconvex subgroup of a hyperbolic group $G$, then $Q$ is hyperbolic.
\end{theorem}

The following allows us to conclude that some relatively hyperbolic groups are hyperbolic by understanding the peripheral subgroups.

\begin{proposition}\thlabel{hyp rel to hyps}\cite[Corollary 2.41]{OsinMemoirs}
    Suppose $(G,\mathbb{P})$ is relatively hyperbolic and each element of $\mathbb{P}$ is also a hyperbolic group. Then $G$ itself is hyperbolic. 
\end{proposition}

\subsection{Coarse Fundamental Group}

\begin{definition}[Coarse Fundamental Group] \cite[Appendix A]{CoarseCH}
    Let $\Gamma$ be a connected graph with a vertex $a_0$ and let $D >0$. Let $\Gamma^D$ be the space obtained by attaching a disk to all edge loops of length $\leq D$ and set 
    \[\pi_1^D(\Gamma,a_0) = \pi_1(\Gamma^D,a_0).\]
    The \emph{$D$--universal cover of $\Gamma$} is the natural preimage of $\Gamma$ in the universal cover of $\Gamma^D$. We say $\Gamma$ is \emph{$D$--simply connected} if $\pi_1^D(\Gamma,a_0) = \{1\}$, equivalently, if $\pi_1(\Gamma,a_0)$ is generated by loops freely homotopic to loops of length at most $D$. If $\Gamma$ is $D$ simply connected for some $D$, we say $\Gamma$ is \emph{coarsely simply connected}.
\end{definition}

\begin{lemma}\cite[proof of III.$\Gamma.2.6$]{BH}\thlabel{Hyperbolic spaces are coarsely simply connected}
    If $\Gamma$ is a $\delta$--hyperbolic graph with integer edge lengths, then $\Gamma$ is $16\delta$--simply connected. 
\end{lemma}

In the case of combinatorial horoballs, we can do even better than this general result. 

\begin{lemma}\thlabel{horoballs are simply connected}
    If $\Gamma$ is a connected graph, then $\sH(\Gamma)$ is $5$--simply connected.
\end{lemma}
\begin{proof}
    In \cite[Definition 3.1]{DehnFillinginRHG}, combinatorial horoballs are defined with $2$--cells which are horizontal triangles, vertical squares, and pentagons. In \cite[Proposition 3.7]{DehnFillinginRHG}, they show this space is simply connected. Thus we need $2$--cells with circumference at most $5$ to make a combinatorial horoball simply connected.
\end{proof}

The following is a local to global criterion for showing a space is hyperbolic which makes use of coarse simple connectivity and local negative curvature.

\begin{theorem}[Coarse Cartan Hadamard]\thlabel{coarseCH}\cite[Theorem A.1]{CoarseCH} Let $\nu >0$ and $\sigma > 10^7\nu$. Let $Z$ be a geodesic metric space. If every ball of radius $\sigma$ in $Z$ is $\nu$--hyperbolic and $Z$ is $10^{-5}\sigma$ simply connected, then $Z$ is $300\nu$--hyperbolic.
\end{theorem}

\subsection{Fillings and the Cohen Lyndon Property}

\begin{definition}\thlabel{Defn Dehn Filling}\cite[Introduction]{OsinFillings}
    Let $G$ be a group. For any subset $S \subset G$ let $\llangle S \rrangle$ be the normal closure of $S$ in $G$. Suppose that $\{H_\lambda\}_{\lambda \in \Lambda}$ is a family of subgroups of $G$ and $\{N_\lambda \triangleleft H_\lambda\}_{\lambda \in \Lambda}$ is a family of normal subgroups in each $H_\lambda$. We call $(G, \{H_\lambda\}_{\lambda \in \Lambda}, \{N_\lambda\}_{\lambda \in \Lambda})$ a \emph{group triple} and let
    \[N = \bigcup_{\lambda \in \Lambda}N_\lambda \quad \quad \overline{G} = G/ \llangle N \rrangle \quad \quad \overline{H}_\lambda = H_\lambda/ N_\lambda.\]
    If $\Lambda$ is finite, we call the group $\overline{G}$ a \emph{Dehn filling} of $G$ and the $N_\lambda$ are called \emph{filling kernels}.
\end{definition}

\begin{definition}
    Given a group $G$ with a finite collection of subgroups $\mathbb{P} = \{P_1,\ldots P_n\}$, we say that a property holds \emph{for all sufficiently long fillings} of $G$ if there is a finite set $F \subset G\setminus \{1\}$ so that whenever $\{N_1,\ldots ,N_n\}$ is a set of filling kernels with $N_i \cap F = \varnothing$ for each $i$, then $\overline{G}$ has the property. 
\end{definition}

The following is the fundamental theorem of relatively hyperbolic Dehn Fillings.

\begin{theorem} \cite[Theorem 1.1]{OsinFillings}
    Let $(G,\{P_i\}_i)$ be relatively hyperbolic. For any finite subset $S \subset G \setminus\{1\}$, the following holds for any sufficiently long filling $\varphi:G \longrightarrow G/K$:
    \begin{enumerate}
        \item for each $i$, $\varphi$ induces an embedding of $P_i/N_i \longrightarrow G/K$ whose image we identify with $P_i/N_i$,
        \item $(G/K, \{P_i/N_i\}_i)$ is relatively hyperbolic, and
        \item $\varphi$ restricted to $S$ is injective. 
    \end{enumerate}
\end{theorem}

By identifying $G$ with the points of depth $0$ in the cusped space and using \thref{horoballs are hyperbolic} to control how the metric on $G$ warps in the cusped space, we can use the third conclusion of the previous theorem to get the following essential tool. 

\begin{lemma}\thlabel{injective on balls centered on cayley graph}\cite[Lemma 2.16]{BoundariesOfDehnFillings}
    Suppose $(G,\mathbb{P})$ is relatively hyperbolic, $Y$ is a Cayley graph for $G$, and $\widehat{Y}$ is a cusped space for $G$ built on $Y$. If $R$ is a positive integer, then for all sufficiently long fillings with kernel $K$, the quotient $p:\widehat{Y} \longrightarrow \widehat{Y}/K$ is injective on any ball of radius $R$ centered on a point in $Y$.
\end{lemma}

It is helpful to know that the normal closure of a collection of filling kernels is the free product of the filling kernels. This definition was first considered in \cite{CohenLyndon} where it was called the \emph{free product of maximally many conjugates} property.

\begin{definition}[Cohen--Lyndon Property]\cite{CohenLyndon}
    A group triple $(G, \{H_\lambda\}_{\lambda \in \Lambda}, \{N_\lambda\}_{\lambda \in \Lambda})$ is called a \emph{Cohen--Lyndon triple} if there exist left transversals $T_\lambda$ of $H_\lambda \llangle N\rrangle$ so that the natural homomorphism
    \[{\Large * }_{\lambda \in \Lambda, t \in T_{\lambda}} tN_\lambda t^{-1} \longrightarrow \llangle N \rrangle\]
    is an isomorphism.
\end{definition}

In \cite[Theorem 2.27]{DGO}, the authors consider the more general setting of a group $G$ with a collection of \emph{hyperbolically embedded subgroups} $\mathbb{P}$ and show that long enough fillings of such a $(G,\mathbb{P})$ form a Cohen--Lyndon triple. Their approach is through \emph{very rotating families}, which are a geometric tool for proving this algebraic property. The following geometric condition is similar to these very rotating families.

\begin{definition}\thlabel{geometrically CH}(Very Translating Condition)\cite[$\S$4]{BoundariesOfDehnFillings}
    Suppose $(G,\{P_i\}_i)$ has a cusp uniform action on a $\delta$--hyperbolic metric graph $X$ and let $\sH_i$ be the horoball stabilized by $P_i$. A set of filling kernels $\{N_i \triangleleft P_i\}_i$ satisfy the \emph{very translating condition} if for each $n \in N_i \setminus \{1\}$ and $x \notin \sH_i$, we have $d(x,n\cdot x) \geq 10^4\delta$. We call the triple $(G,\{P_i\}_i, \{N_i\}_i)$ a \emph{very translating triple}.
\end{definition}

The following is immediate from \cite[Theorem 4.8]{BoundariesOfDehnFillings} and is stated for a single peripheral subgroup in \cite{Drilling}.

\begin{theorem}\cite[Theorem 11.2]{Drilling}\thlabel{very translating implies CH}
    If $(G,\{P_i\}_{i},\{N_i\}_i)$ is a very translating triple, then $(G,\{P_i\}_i,\{N_i\}_i)$ is a Cohen--Lyndon triple.
\end{theorem}

Just as a finite index subgroup of a relatively hyperbolic group inherits a peripheral structure from the larger group, any Dehn filling of the larger group induces a Dehn filling of the subgroup. The next definition is central to this work.

\begin{definition}[Induced Filling]\thlabel{induced filling}
    Let $(G,\{P_i\}_i)$ be a relatively hyperbolic group with $H<G$ a finite index subgroup. Choose subgroups $P_j' = H \cap P_{i(j)}^{g_j}$ so that $\{P_j'\}_j$ is the induced peripheral structure for $H$. If $\{N_i \triangleleft P_i\}_i$ is a set of filling kernels for $G$, the \emph{induced filling kernels for $H$} are the subgroups $N_j' = H \cap N_{i(j)}^{g_j} \triangleleft P_j'$. The \emph{induced filling of $H$ by $\{N_i\}_i$} is the quotient $H/ \llangle N_j'\rrangle_H$. We call $(H,\{P_j'\}_j,\{N_j'\}_j)$ the \emph{induced group triple for $H$}. The \emph{filled parabolics of $H$} are the groups $P_j'/N_j'$.
\end{definition}

If a group $G$ acts on a metric space, then so does each of its subgroups, and the restriction of this action to subgroups inherits many of the properties of the larger action. This observation is the main strategy of this paper, and the next result is one of our first examples of this.

\begin{corollary}\thlabel{finite index subgroup of geom CH}
    If $(G,\{P_i\}_i,\{N_i\}_i)$ is very translating group triple and $H$ is a finite index subgroup of $G$, then the induced group triple for $H$ is a Cohen--Lyndon triple.
\end{corollary}
\begin{proof}
    The definition of a very translating group triple \thref{geometrically CH} includes an action of the triple $(G,\{P_i\}_i,\{N_i\}_i)$ on a metric space with certain properties. The induced group triple for $H$ acts on the same space and inherits all the necessary metric properties from $(G,\{P_i\}_i,\{N_i\}_i)$ to be a very translating group triple, so the result follows from \thref{very translating implies CH}.
\end{proof}

We collect a few more results on induced fillings which we use in later sections. We state these where the larger group has a single peripheral subgroup for notational convenience and because it matches the situation of a generalized drilling (see \thref{defn generalized}) but the same proof holds for multiple peripheral subgroups. 

\begin{lemma}\thlabel{induced filling for normal subgroups}
    Suppose $(G,P)$ is relatively hyperbolic, $N \triangleleft P$ is a filling kernel, $\overline{P} = P/N$, $K = \llangle N \rrangle _G$ and $\overline{G}= G/K$. Fix a finite index subgroup $H$ of $G$ and let $(H,\{P_i\}_i,\{N_i\}_i)$ be the induced group triple for $H$. Let $\overline{P}_i = P_i/N_i$, $K_N = \llangle N_i \rrangle_{H}$, and $\overline{H} = H/K_N$.
    \begin{enumerate}
        \item If $H$ is normal in $G$, then we can consider either $N$ or $H\cap N$ as a filling kernel for $G$. Both of these fillings induce the same filling of $H$.
        \item Suppose the natural quotient $G \longrightarrow \overline{G}$ induces an embedding of $\overline{P}$ into $\overline{G}$. Then the $\{\overline{P}_i\}_i$ embed in $\overline{H}$ and embed in $\overline{G}$ as finite index subgroups of $\overline{P}$. If $\overline{P}$ is hyperbolic, then so is each filled parabolic of $H$.
        \item If $(G,P,N)$ is a generalized drilling of $\overline{G}$ along some $Q < G$, then the $\{\overline{P}_i\}_i$ embed in $\overline{H}$ and embed in $\overline{G}$, and each $\overline{P}_i$ is hyperbolic. If $Q \cong\ZZ$, then $\overline{P}_i \cong \ZZ$ for each $i$.
    \end{enumerate}
\end{lemma}
\begin{proof}
    As in \thref{induced peripheral}, the induced peripheral structure for $H$ involves a choice of elements $g_i \in G$ so that $P_i := H \cap P^{g_i}$ and $\{P_i\}_i$ is a collection of $H$--conjugacy class representatives of subgroups in $\{H \cap P^g\;|\; g \in G\}$. By \thref{induced filling}, we have $N_i:= H \cap N^{g_i}$, and the filling kernel for $H$ induced by $N$ is $K_N:= \llangle N_i \rrangle_H$.
    
    To prove $(1)$, suppose $H$ is normal in $G$. Then $H\cap N$ is normal in $P$, so we can indeed consider $H\cap N$ as a filling kernel for $G$. Again using \thref{induced filling}, the induced filling of $H$ by $H \cap N$ is the quotient of $H$ by $K_{H \cap N} := \llangle H \cap (H\cap N)^{g_i}\rrangle_H$. Note that $H\cap (H\cap N)^{g_i} \leq H\cap N^{g_i}$, and $K_N$ is the normal closure in $H$ of the possibly larger subgroups $H\cap N^{g_i}$, so even without the assumption of normality on $H$ we have $K_{H\cap N} \leq K_N$. When $H$ is normal, we have the reverse inclusion
    \[H \cap (H\cap N)^{g_i} = H \cap H^{g_i} \cap N^{g_i} = H \cap N^{g_i},\]
    which implies $K_N = K_{H\cap N}$. The filling kernels coincide, so the fillings of $H$ are the same. This proves $(1)$.

    For $(2)$, observe that the image of $P$ in $\overline{G}$ is isomorphic to $P /P\cap K$, so the assumption that $\overline{P}$ embeds in $\overline{G}$ is equivalent to $P \cap K = N$. For any $g \in G$, we can conjugate this equation and use the normality of $K$ to see that $P^g \cap K = N^g$. Because $K_H$ is the normal closure in $H$ of some subgroups of conjugates of $N$, $K_H < K$. As with $G$ and $P$, the quotient $H \longrightarrow \overline{H}$ induces the quotient $P_i \longrightarrow \overline{P}_i$ if and only if $P_i \cap K_H = N_i$. To show that indeed $P_i \cap K_H = N_i$ for any $i$, suppose $p \in P_i \cap K_H$. Then $p \in P_i <P^{g_i}$, $p \in K_H<K$, and clearly $p \in H$, so 
    \[p \in P^{g_i} \cap K \cap H = N_i \cap H =:N,\]
    where the first equality uses that $N = P \cap K$. Thus the $\overline{P}_i$ embed in $\overline{H}$. To see they are isomorphic to finite index subgroups of $\overline{P}$, observe that $P_i$ is finite index in $P^{g_i}$ because $H$ is finite index in $G$. The image of $P_i$ in $\overline{G}$ is isomorphic to $P_i / P_i \cap K$, and we claim $P_i\cap K = N_i$. Using the definition of $P_i$ and $N_i$ together with $P \cap K = N$, we have 
    \[P_i \cap K = H \cap P^{g_i} \cap K = H \cap N^{g_i} = N_i\]
    Thus $\overline{P}_i$ also embeds into $\overline{G}$ as a finite index subgroup of (a conjugate of) $\overline{P}$. If in addition $\overline{P}$ is hyperbolic, then each $\overline{P}_i$ must be hyperbolic because it is a finite index subgroup of a hyperbolic group.

    For $(3)$, note that by \thref{defn generalized}, a generalized drilling includes an identification of $\overline{P}$ with $Q$ and the quotient $G \longrightarrow \overline{G}$ induces this identification. Thus $(3)$ is a special case of $(2)$, and the last part follows from the fact that the only finite index subgroups of $\ZZ$ are isomorphic to $\ZZ$.
\end{proof}

\subsection{$PD(n)$ Groups}\label{PDn stuff}

If $G$ is the fundamental group of a closed aspherical $n$--manifold $M$, then $M$ is a $K(G,1)$. Because $M$ is a manifold, its cohomology -- and thus the cohomology of $G$ -- satisfies a Poincar\'e duality relationship. $PD(n)$ groups are groups whose cohomology satisfies a similar duality relationship:

\begin{definition}\cite{JohnsonWall}
    A group $G$ is $PD(n)$, or \emph{Poincar\'e duality group of dimension $n$}, if
    \begin{enumerate}
        \item $G$ is type $FP$,
        \item $H^n(G,\ZZ G) \cong \ZZ$, and 
        \item $H^i(G,\mathbb{Z}G) = 0$ for $i \neq n$.
    \end{enumerate}
\end{definition}

The theory of duality in groups is deep and complex, but we do not deal with any specifics of group cohomology or related topics in this work. The interested reader can refer to \cite{CohomGroupsBrown}. For hyperbolic groups, Bestvina and Mess have shown that, for hyperbolic groups, being a $PD(n)$ group puts strong conditions on the boundary:

\begin{theorem}
    \thlabel{bestvina boundary calculator}\cite[Corollary 1.3 (c)]{BestvinaMess}\cite[Theorem 2.8]{LocalHomOfBoundaries}
    If $G$ is a hyperbolic group, then $G$ is a $PD(n)$ group if and only if $G$ is torsion free and $\partial G$ is an integral \v{C}ech cohomology sphere of dimension $n-1$.
\end{theorem}

In particular, Bestivna showed the following about hyperbolic $PD(3)$ groups. This covers the case $n=3$, which is our focus.

\begin{theorem}\thlabel{bestvina boundary calculator n=3}\cite[Remark 2.9]{LocalHomOfBoundaries}
If $G$ is a hyperbolic $PD(n)$ group for $n \leq 3$, then $\partial G \cong S^{n-1}$.
\end{theorem}

If $M$ is a manifold with boundary, then the relative cohomology of the pair $(M,\partial M)$ satisfies \emph{Lefschetz} duality. For such an $M$, we can take two copies of $M$ and identify components of $\partial M$ to build a manifold without boundary called the \emph{double} of $M$. To generalize this to groups, $(M,\partial M)$ is generalized to the notion of a $PD(n)$ pair, and the following definition is an algebraic analog of constructing the double of $M$.

\begin{definition}
    Suppose $G$ is a group and $\mathbb{P}$ is a finite collection of subgroups. Let $G,G'$ be isomorphic copies of $G$ with an isomorphism $\varphi:G \longrightarrow G'$. Let $\mathscr{G}$ be a graph of groups with $2$ vertices $v_G,v_{G'}$ whose vertex groups are $G$ and $G'$, and with one edge $e_P$ between this pair of vertices for each $P \in \mathbb{P}$. Each $e_P$ has edge group $P$ and includes into $G,G'$ using the identity and $\varphi$ respectively. Then $\pi_1(\mathscr{G})$ is the \emph{double of $(G,\mathbb{P})$ along $\varphi$}.  
\end{definition}

\begin{definition}\cite[Corollary 8.5]{BieriEckmann}
    If $G$ is a group and $\mathbb{P}$ is a finite collection of subgroups, then $(G,\mathbb{P})$ is a $PD(n)$ pair if the double of $(G,\mathbb{P})$ along the identity is a $PD(n)$ group.     
\end{definition}

The following result of Tsushiku and Walsh is the analog of \thref{bestvina boundary calculator n=3} for relatively hyperbolic groups.

\begin{theorem}\thlabel{TWPD3}\cite[Corollary 3, Theorem 4]{TsushikuWalsh}
    Let $(G,\mathbb{P})$ be a relatively hyperbolic group. Then $(G,\mathbb{P})$ is a $PD(3)$ pair if and only if $G$ is torsion free and $\partial_{\mathbb{P}}G \cong S^2$.
\end{theorem}

More generally, we note that Manning and Wang \cite[Theorem 1.3]{S2boundaryMW} show that the analogous statement of \thref{bestvina boundary calculator} holds for a relatively hyperbolic pair $(G,\mathbb{P})$ if $G$ has type $F$, equivalently, if $G$ has a finite $K(G,1)$. By \cite[Theorem 0.1]{BoundariesTypeF}, $G$ has type $F$ if it is torsion free and each peripheral subgroup has type $F$. Since we are focused on the case where $n=3$, \thref{TWPD3} is all we need.

This next lemma connects the cohomological and boundary properties of a relatively hyperbolic group with all of its finite index subgroups. 

\begin{lemma}\thlabel{cohom boundary connector}
    Suppose $(G,\mathbb{P})$ is a torsion free relatively hyperbolic group and $H < G$ is a finite index subgroup with induced peripheral structure $\mathbb{P}_H$. The following are equivalent. 
    \begin{enumerate}
        \item $(G,\mathbb{P})$ is a $PD(3)$ pair. 
        \item $\partial_{\mathbb{P}}G \cong S^2$.
        \item $(H,\mathbb{P}_H)$ is a $PD(3)$ pair.
        \item $\partial_{\mathbb{P}_H}H \cong S^2$.
    \end{enumerate}
\end{lemma}
\begin{proof}
    Recall that $(H,\mathbb{P}_H)$ is a relatively hyperbolic group and $\partial_\mathbb{P}G \cong \partial_{\mathbb{P}_H}H$ by \thref{fin index subgroup is rel hyp}, essentially because $H$ acts geometrically finitely whenever $G$ does. This makes $(2)$ and $(4)$ equivalent. The equivalences $(1) \Longleftrightarrow (2)$ and $(3) \Longleftrightarrow (4)$ are both the previous theorem. That completes the proof, but we also note that $(1)$ and $(3)$ are equivalent without any hyperbolicity assumptions using \cite[Theorem 7.6]{BieriEckmann}. 
\end{proof}

Petrosyan and Sun prove many cohomological results about Dehn fillings in \cite{CohomOfFillings}. They work in the more general context where $(G,\mathbb{P})$ is a group with a collection of hyperbolically embedded subgroups, and prove various results for long enough fillings. In particular they show that, with some assumptions, long enough fillings of a $PD(n)$ pair produce a $PD(n)$ group. The following gives an effective version of this statement by listing the necessary properties of such a filling. This result does not appear verbatim in their work, but follows from it.

\begin{theorem}\thlabel{filled group is PDn}\cite[Theorem 5.2]{CohomOfFillings}
    Suppose $(G,\{P_i\}_i)$ is relatively hyperbolic and a $PD(n)$ pair. Suppose a set of filling kernels $\{N_i\}_i$ satisfy the following:
    \begin{enumerate}
        \item $(G,\{P_i\}_i, \{N_i\}_i)$ is a Cohen--Lyndon triple,
        \item $N_i \cong \ZZ$ for each $i$,
        \item $\overline{P_i} := P_i/N_i$ is a $PD(n-2)$ group and embeds into $\overline{G}$,
        \item $(\overline{G},\{\overline{P}_i\}_i)$ is relatively hyperbolic.
    \end{enumerate}
    Then $\overline{G}$ is a $PD(n)$ group. 
\end{theorem}
\begin{proof}
    This statement differs from the statement of \cite[Theorem 5.2]{CohomOfFillings} in two ways. 
    
    First, \cite[Theorem 5.2]{CohomOfFillings} supposes that we have a collection of hyperbolically embedded subgroups in $G$, rather than a relatively hyperbolic group. By \cite[Proposition 4.28]{DGO}, if $(G,\mathbb{P})$ is relatively hyperbolic, then $\mathbb{P}$ is hyperbolically embedded in $G$, so our assumption is stronger than Petrosyan and Sun's. 
    
    Second, \cite[Theorem 5.2]{CohomOfFillings} is stated only for long enough fillings and doesn't specify the properties of the filling. The properties in the assumption are exactly what is needed to make the filling long enough; The proof of \cite[Theorem 5.2]{CohomOfFillings} given relies only on commutative algebra arguments and the results of \cite[$\S$4]{CohomOfFillings}, which in turn only needs fillings which satisfy the assumed properties. Explicitly, 
    \begin{itemize}
        \item \cite[Lemma 4.1 through Corollary 4.7]{CohomOfFillings} are commutative algebra arguments with Cohen--Lyndon triples,
        \item \cite[Theorem 4.8, Corollary 4.9]{CohomOfFillings} assumes that the $\overline{P}_i$ are hyperbolically embedded in $\overline{G}$. We assume these subgroups embed in $\overline{G}$ and that $(\overline{G},\{\overline{P}_i\}_i)$ is relatively hyperbolic, which is stronger than assuming the $\overline{P}_i$ are hyperbolically embedded, as explained above, and
        \item \cite[Theorem 4.10]{CohomOfFillings} combines the preceding results in this section. 
    \end{itemize} 
\end{proof}

%% file: Sections/3_BuildingBranchedCovers.tex
\section{A Metric Criterion}\label{sec 3}

In the introduction, we explain how to construct branched covers of a $3$--manifold by drilling along a geodesic, taking a finite cover of the drilled manifold, and filling the finite cover. In analogy with this strategy, we construct algebraic branched covers by starting with a generalized drilling $(\widehat{G},P,N)$ of a hyperbolic group $G$, passing to a finite index subgroup $\widehat{H}$, and filling $\widehat{H}$ using the induced filling from $N$ to get a group $H$. To conclude that $H$ is indeed hyperbolic and therefore a branched cover of $G$, we prove a metric criterion \thref{main metric lemma try2} for when a Dehn filling of a relatively hyperbolic group gives a relatively hyperbolic group. Combined with \thref{hyp rel to hyps}, this gives a tool for showing a filling is hyperbolic. It is essentially an effective version of \cite[Proposition 2.2]{QCandDehnFillings} and \cite[Proposition 2.3]{Agol_2009}, and uses the same proof strategy. We need an effective statement so that we can apply it to specific fillings later.

The following definition is the setting in which we prove \thref{main metric lemma try2}. It combines the flexibility of a cusp uniform action in \thref{gromovs defn} with the ease of working with a free action on a graph and combinatorial horoballs.

\begin{definition}\thlabel{defn pseudo cusped space}
    A \emph{pseudo--cusped space} for a group pair $(G,\{P_i\}_i)$ is a pair of simplicial graphs $(\widehat{Y},\widecheck{Y})$ so that 
    \begin{enumerate}
        \item $G$ acts freely and cocompactly on $\widecheck{Y}$,
        \item $\widehat{Y}$ is $\theta_1$--hyperbolic and is constructed from $\widecheck{Y}$ by attaching combinatorial horoballs to a $G$--invariant collection $\mathbb{S}$ of connected subgraphs of $\widecheck{Y}$ called \emph{horospheres},
        \item the $\{P_i\}_i$ form a collection of conjugacy class representatives for $G$--stabilizers of horospheres.
    \end{enumerate}
\end{definition}

\begin{lemma}\thlabel{pseudo cusped implies rel hyp}
    If a group pair $(G,\{P_i\}_i)$ has a pseudo--cusped space $(\widehat{Y},\widecheck{Y})$, then $(G,\{P_i\}_i)$ is relatively hyperbolic with Bowditch boundary $\partial \widehat{Y}$.
\end{lemma}
\begin{proof}
    By \thref{finally write it down}, the combinatorial horoballs in \thref{defn pseudo cusped space} are coarse horoballs, and it is clear that the action of $(G,\{P_i\}_i)$ on $\widehat{Y}$ is cusp--uniform.
\end{proof}

\begin{proposition}\thlabel{main metric lemma try2}
    Let $(\widehat{Y},\widecheck{Y})$ be a pseudo--cusped space for the group pair $(G,\{P_i\}_i)$. Let $\{N_i\triangleleft P_i\}_i$ be a collection of filling kernels for $G$ with $K = \llangle N_i\rrangle_G$, and let $\varphi:G \longrightarrow \overline{G} = G/K$ be the natural quotient. Suppose that \begin{enumerate}
        \item the hyperbolicity constant $\theta_1$ of $\widehat{Y}$ satisfies $\theta_1 \geq \theta_0$, where $\theta_0$ is the universal hyperbolicity constant of combinatorial horoballs in \thref{horoballs are hyperbolic}, and
        \item if $B$ is a ball of radius $6\cdot 10^7\theta_1$ in $\widehat{Y}$ around a point $z \in \widecheck{Y}$, then the natural quotient $\widehat{Y}\longrightarrow \leftQ{\widehat{Y}}{K}$ is injective on $B$. 
    \end{enumerate}
    Then $(\overline{G},\{\varphi(P_i)\}_i)$ is relatively hyperbolic.
\end{proposition}

\begin{assumption}
    For the rest of this subsection, fix notation as in the statement of the previous proposition, as well as the following.
    \begin{enumerate}
        \item Let $\widecheck{\Gamma} = \leftQ{\widecheck{Y}}{K}$ with $p:\widecheck{Y}\longrightarrow \widecheck{\Gamma}$ the natural quotient map. Note that $p$ and $\varphi$ intertwine the actions of $G,\overline{G}$: for $z \in \widecheck{Y}, g \in G$, we have $p(g\cdot z) = \varphi(g)\cdot p(z)$. 
        \item Construct $\widehat{\Gamma}$ from $\widecheck{\Gamma}$ by attaching a combinatorial horoball to the image of each horosphere from $\widecheck{Y}$, and extend $p$ to a depth preserving map $\widehat{Y}\longrightarrow \widehat{\Gamma}$.
        \item Let $\sS_i$ be the horosphere of $\widecheck{Y}$ stabilized by $P_i$ and let $\sH_i$ be the horoball of $\widehat{Y}$ containing $\sS_i$. 
        \item Let $\overline{\sS}_i = p(\sS_i)$, and let $\overline{\sH}_i = p(\sH_i)$ be the horoball of $\widehat{\Gamma}$ attached to $\overline{\sS}_i$. 
        \item Increasing $\theta_1$ if necessary, assume $16\theta_1 \geq 5$.
        \item Let $\sigma=10^7\theta_1$.
    \end{enumerate}
\end{assumption}

\begin{lemma}\thlabel{diagram star}
    The following diagram commutes. Each map is a covering map, and $q,p,F$ have deck groups $G, K, G/K$ respectively.  

    \begin{center}
\begin{tikzcd}
\widecheck{Y} \arrow[rr, "q"] \arrow[rd, "p"'] &                                    & \leftQ{\widecheck{Y}}{G} \\
                                               & \widecheck{\Gamma} \arrow[ru, "F"'] &            
\end{tikzcd}
\end{center}
\end{lemma}
\begin{proof}
    If a group acts freely on a metric graph, the quotient is a covering map. $G$ acts freely on $\widecheck{Y}$ by assumption, so $K$ does as well, which implies the result for $q,p$. The normality of $K$ implies $\overline{G}$ acts freely on $\widecheck{\Gamma}$, implying the result for $F$. 
\end{proof}

We show that $\widehat{\Gamma}$ is hyperbolic using the Coarse Cartan Hadamard \thref{coarseCH}, then show that $(\widehat{\Gamma},\widecheck{\Gamma})$ is a pseudo--cusped space for $(\overline{G},\{\overline{P}_i\}_i)$, and conclude.

\begin{lemma}\thlabel{widehatGamma is simply connected}
    $\widehat{\Gamma}$ is $16\theta_1$--simply connected.
\end{lemma}
\begin{proof}
    Suppose $\overline{\sH}$ is a horoball of $\widehat{\Gamma}$ with horosphere $\overline{\sS}$ and let $\gamma$ be a loop in $\overline{\sS}$. Because $16\theta_1 > 5$ and horoballs are $5$--simply connected by \thref{horoballs are simply connected}, $\gamma$ is nullhomotopic in $\pi_1^{16\theta_1}(\overline{\sH})$. By definition $\widehat{\Gamma}$ is $\widecheck{\Gamma}$ with horoballs added to these horospheres, and adding these horoballs kills the $16\theta_1$ fundamental group of that horosphere. Therefore to show $\pi_1^{16\theta_1}(\widehat{\Gamma})$ is trivial, it is enough to show that it is normally generated by loops contained in horospheres. By the previous lemma and standard covering space theory, there is a short exact sequence of groups
    \[1 \longrightarrow \pi_1(\widecheck{Y}) \longrightarrow \pi_1(\widecheck{\Gamma}) \longrightarrow K \longrightarrow 1.\]
    Because $\widehat{Y}$ is $\theta_1$--hyperbolic, it is $16\theta_1$--simply connected by \thref{Hyperbolic spaces are coarsely simply connected}.
    Adding these horoballs kills all of $\pi_1^{16\theta_1}(\widecheck{Y})$ so the Seifert Van Kampen Theorem implies $\pi_1^{16\theta_1}(\widecheck{Y})$ is generated by loops which are freely homotopic to loops contained in elements of $\mathbb{S}$. Such loops map to loops in the horospheres of $\widecheck{\Gamma}$, which become nullhomotopic in $\widehat{\Gamma}^{16\theta_1}$. 
    
    Consider some $N_i$ and fix some $x \in \sS_i$. Because $N_i$ stabilizes $\sS_i$ and $\sS_i$ is connected, we can connect $x$ to $n\cdot x$ by a path in $\sS_i$ for any $n \in N_i$. This path maps to a loop in $\overline{\sS}_i$ through $p$ which represents an element of $\pi_1(\widecheck{\Gamma})$ which is not in the image of $\pi_1(\widecheck{Y})$. Because $K = \llangle N_i\rrangle$, this implies that the contribution of $K$ to $\pi_1^{16\theta_1}(\widecheck{\Gamma})$ is normally generated by loops which are freely homotopic to loops in the horospheres of $\widecheck{\Gamma}$, which again become nullhomotopic in $\widehat{\Gamma}^{16\theta_1}$.
\end{proof}

The previous lemma handles the coarsely simply connected requirement of \thref{coarseCH}, and we show the local hyperbolicity requirement using the large injectivity radius. Roughly speaking, the following upgrades a local homeomorphism to a local isometry.

\begin{lemma}\thlabel{upgrade to isometry}\cite[Lemma 7.8]{Drilling}
    Let $A,B$ be metric graphs, and let $B_R(p)$ be a ball of radius $R$ in $A$ for some integer $R >3$. Suppose we have an injective simplicial map $f:B_R(p) \longrightarrow B$ which is a local homeomorphism at every point of $B_{R-1}(p)$. Then $f$ restricts to an isometry $f':B_{R/3}(p) \longrightarrow B_{R/3}(f(p))$. 
\end{lemma}

\begin{lemma}\thlabel{modelling for criterion}
    Any ball of radius $\sigma_0$ in $\widehat{\Gamma}$ is isometric to a ball in a $\theta_1$--hyperbolic space.
\end{lemma}
\begin{proof}
    Fix $z \in \widehat{\Gamma}$. By \thref{upgrade to isometry}, it suffices to show that $B = B(z,3\sigma_0)$ injects into a $\theta_1$--hyperbolic graph. There are two cases depending on whether or not $z$ is deep in a horoball.

    If $z$ is deeper than $\sigma_0$ in a horoball $\overline{\sH}$ of $\widehat{\Gamma}$, then the inclusion $B \longrightarrow \overline{\sH}$ is injective. All combinatorial horoballs are $\theta_0$--hyperbolic and $\theta_0 \leq \theta_1$ by assumption, so this completes this case.

    Otherwise, $z$ is at depth less than $3\sigma_0$. If $z$ is at positive depth, let $z'$ be the point above $z$ at depth $0$, and if $z$ is at depth $0$, let $z' =z$. Either way, there is some $z' \in \widecheck{\Gamma}$ with $d_{\widehat{\Gamma}}(z,z')\leq 3\sigma_0$ so that $B = B(z,3\sigma_0) \subset B(z',6\sigma_0)$. Let $\widetilde{z}$ be a point of $\widecheck{Y}$ so that $p(\widetilde{z}) = z'$. By assumption, $p$ is injective on $B(\widetilde{z},6\sigma_0)$, so we can define an injective inverse $\iota:B(z',6\sigma_0) \longrightarrow B(\widetilde{z},6\sigma_0)$. The restriction of $\iota$ to $B$ gives the desired injection of $B$ into a $\theta_1$--hyperbolic space.
\end{proof}

\begin{lemma}\thlabel{widehatGamma is hyperbolic}
    The space $\widehat{\Gamma}$ is $\theta_2 := 300\theta_1$--hyperbolic. 
\end{lemma}
\begin{proof}
    \thref{widehatGamma is simply connected} shows that $\widehat{\Gamma}$ is $16\theta_1$--simply connected, and \thref{modelling for criterion} shows that every ball of radius $\sigma_0$ in $\widehat{\Gamma}$ is isometric to a ball in a $\theta_1$--hyperbolic space. Since $10^5(16\theta_1) \leq 10^7\theta_1$, we have $\sigma_0 \geq\max \{10^7\theta_1,10^5(16\theta_1)\}$, so this result follows from the Coarse Cartan Hadamard \thref{coarseCH}.
\end{proof}

\begin{lemma}\thlabel{stabilizer of image of horoball is image of stabilizer}
    The set $\{\overline{\sH}_i\}_i$ is a collection of representatives for distinct $\overline{G}$--orbits of horoballs in $\widehat{\Gamma}$. If $\overline{\sH}$ is a horoball of $\widehat{\Gamma}$, then $\textrm{Stab}_{\overline{G}}(\overline{\sH})$ is a conjugate of some $\varphi(P_i)$ which acts geometrically on $\overline{\sS}_i$.
\end{lemma}
\begin{proof}
    Essentially by construction, each horoball of $\widehat{\Gamma}$ is the image of a horoball of $\widehat{Y}$. The $\{\sH_i\}_i$ form a set of representatives of $G$--orbits of horospheres in $\widehat{Y}$ and $p$ and intertwines the actions of $G,\overline{G}$ with $\varphi$, so $\{\overline{\sH}_i\}_i$ meets each $\overline{G}$--orbit of horosphere in $\widehat{\Gamma}$ at least once. We can phrase this more explicitly using elements; If $\overline{\sH} \subset\widehat{\Gamma}$ is a horoball, then $\overline{\sH} = p(\sH)$ for some horoball $\sH \subset \widehat{Y}$. There is some $g \in G$ and $i$ so that $\sH = g\sH_i$, so 
    \[\overline{\sH} = p(g\sH_i) = \varphi(g)p(\sH_i) = \varphi(g) \overline{\sH}_i. \]
    The $\{\overline{\sH}_i\}_i$ represent \emph{distinct} $\overline{G}$--orbits because $p$ collapses $K$--orbits and the $\sH_i$ are in distinct $G$--orbits, hence different $K$--orbits. More precisely, if $g\overline{\sH}_i = \overline{\sH}_j$ for some $g \in \overline{G}$, then there is some $\widetilde{g} \in G$ so that $\varphi(\widetilde{g}) = g$. Then $p(\widetilde{g}\sH_i) = p(\sH_j)$, and because $p$ is the quotient by $K$, there must be some $k \in K$ so that $k\widetilde{g}\sH_i = \sH_j$. The $\{\sH_i\}_i$ represent different $G$--orbits, so this implies $i = j$. This proves the first assertion.

    Using the first assertion, to prove the second assertion it is enough to show that $\textrm{Stab}_{\overline{G}}(\overline{\sH}_i) = \varphi(P_i)$. It is clear that $\varphi(P_i)\subset \textrm{Stab}_{\overline{G}}(\overline{\sH}_i)$, so we only need to show the reverse containment, and this is essentially a path lifting argument. 
    Suppose $g \in \textrm{Stab}_{\overline{G}}(\overline{\sH}_i)$, $x \in \overline{\sS}_i$, and $\widetilde{x} \in \sS_i$ so that $p(\widetilde{x}) = x$. Let $\gamma$ be a path in $\overline{\sS}_i$ connecting $x, g \cdot x$, and lift $\gamma$ to a path in $\sS_i$ from $\widetilde{x}$ to some $y \in \sS_i$ with $p(y) = g\cdot x$. 
    Using the diagram in \thref{diagram star}, $x,g\cdot x$ are identified in $\leftQ{\widecheck{\Gamma}}{\overline{G}} = \leftQ{\widecheck{Y}}{G}$, so $\widetilde{x},y$ are identified by some element of $G$. Thus there is some element $\widetilde{g} \in \varphi^{-1}(g)$ so that $\widetilde{g} \cdot \widetilde{x} = y$. Because $\widetilde{x}, \widetilde{g}\cdot \widetilde{x}$ are both in $\sS_i$ and the horospheres are disjoint, this implies $\widetilde{g} \in \textrm{Stab}_G(\sS_i) = P_i$. Using that the maps in the diagram of \thref{diagram star} are $D$--covering maps, this $\widetilde{g}$ descends to a deck transformation of $F$. Such deck transformations are determined by their action on a single point, and the descendent of $\widetilde{g}$ clearly takes $x$ to $g\cdot x$, so $\widetilde{g}$ is a lift of $g$. Since $g$ was arbitrary, this shows the reverse inclusion promised above. 
    
    It remains to show $\varphi(P_i)$ acts geometrically on $\overline{\sS}_i$. It acts freely because it acts by deck transformations. $P_i$ acts cocompactly on $\sS$ because $G$ acts cocompactly on $\widecheck{Y}$ and $P_i = \textrm{Stab}_G(\sS_i)$, and
    \[\leftQ{\sS_i}{P_i} = \leftQ{\overline{\sS}_i}{\varphi(P_i)}\]
    by the commutativity of the diagram in \thref{diagram star}. Thus $\varphi(P_i)$ acts cocompactly on $\overline{\sS}_i$.  
\end{proof}

We are ready to complete the goal of this subsection. 

\begin{proof}[Proof of \thref{main metric lemma try2}]
    We claim that $(\widehat{\Gamma}, \widecheck{\Gamma})$ is a pseudo--cusped space for $(\overline{G},\{\varphi(P_i)\}_i)$. Indeed, the action of $\overline{G}$ on $\widecheck{\Gamma}$ is free because $\widehat{\Gamma} = \leftQ{\widecheck{Y}}{K}$, and the action is cocompact because $\leftQ{\widecheck{\Gamma}}{\overline{G}} = \leftQ{\widecheck{Y}}{G}$ is compact. Further, $\widehat{\Gamma}$ is constructed by attaching combinatorial horoballs to $\widecheck{\Gamma}$ along the images of horospheres from $\widecheck{Y}$, and the $\overline{\sH}_i$ are a family of orbit representatives for horoballs with stabilizers $\overline{P}_i$ by \thref{stabilizer of image of horoball is image of stabilizer}. Finally, $\widehat{\Gamma}$ is hyperbolic for some constant by \thref{widehatGamma is hyperbolic}.
\end{proof}

\section{Building Branched Covers}\label{sec building}
\subsection{Deep enough subgroups and fillings give branched covers}

We restate and prove the effective version of our main theorem.

\effectivemain*
\begin{proof}
    A pseudo--cusped space for $(\widehat{G},P)$ is a pseudo--cusped space for $(\widehat{H},\mathbb{P}_{\widehat{H}})$. Let $\mathbb{P}_H$ be the set of filled peripherals for $\widehat{H}$. By \thref{induced filling for normal subgroups} (3), each element of $\mathbb{P}_H$ embeds in $H$ and is hyperbolic.
    
    For the first assertion, the conclusion of \thref{main metric lemma try2} is that $(H,\mathbb{P}_H)$ is relatively hyperbolic. As in the previous paragraph, these images are hyperbolic groups, so $H$ is hyperbolic by \thref{hyp rel to hyps} and $H$ is a branched cover of $G$. If we also assume that $(\widehat{G},P)$ is a $PD(3)$ pair, $N \cong Q\cong \mathbb{Z}$, and the action of $K_{\widehat{H}}$ satisfies the very translating condition, then we can apply \thref{filled group is PDn} to the filling of $\widehat{H}$. To explicitly check the requirements of \thref{filled group is PDn}, note that 
    \begin{enumerate}\addtocounter{enumi}{-1}
        \item $(\widehat{G},P)$ is a $PD(3)$ pair by assumption,
        \item because the action of $K_{\widehat{H}}$ is very translating, \thref{very translating implies CH} implies that $\widehat{H}$, $\mathbb{P}_{\widehat{H}}$ and the collection of induced filling kernels forms a Cohen--Lyndon triple,
        \item $N \cong \ZZ$ by assumption,
        \item Each element of $\mathbb{P}_{H}$ is isomorphic to $\ZZ$ by \thref{induced filling for normal subgroups} (3), and $\ZZ$ is a $PD(1)$ group because it is the fundamental group of the only compact 1--dimensional manifold,
        \item $(H,\mathbb{P}_H)$ is relatively hyperbolic by our application of \thref{main metric lemma try2}.
    \end{enumerate}
    Hence $H$ is a $PD(3)$ group and \thref{bestvina boundary calculator n=3} implies $\partial H \cong S^2$. 

    For the second assertion, if the action of $(\widehat{G},P)$ and $K$ on $(\widehat{Y},\widecheck{Y})$ satisfies the assumptions of \thref{main metric lemma try2}, then so does the action of $(\widehat{H},\mathbb{P}_{\widehat{H}})$ and $K_{\widehat{H}}$. Similarly, if the action of $(\widehat{G},P)$ and $K$ on $(\widehat{Y},\widecheck{Y})$ is very translating, then so is the action of $(\widehat{H},\mathbb{P}_{\widehat{H}})$, as in the proof of \thref{finite index subgroup of geom CH}. Therefore this case is a special case of the first assertion. If we add the further assumptions to $(\widehat{G},P)$, $N$, and $Q$, then again \thref{filled group is PDn} implies $H$ is $PD(3)$ group and \thref{bestvina boundary calculator n=3} implies $\partial H \cong S^2$. 
\end{proof}

We can now derive the following theorem by producing a suitable action for every input and applying the effective version above.

\MainTheorem*

\begin{proof}
    Fix a relatively hyperbolic group $(\widehat{G},P)$. We construct a suitable pseudo--cusped space $(\widehat{Y},\widecheck{Y})$ for $(\widehat{G},P)$ and choice of $F$, and then apply the previous theorem. 

    Choose a finite generating set $S$ for $\widehat{G}$ so that $P = \langle S \cap P\rangle$. Let $\textrm{Cay}$ by the Cayley graph of $\widehat{G}$ with respect to $S$ and note that the choice of $S$ implies the cosets of $P$ are connected subgraphs of $\textrm{Cay}$. Let $\widehat{Y}$ be a cusped space for $(\widehat{G},P)$ built on $\textrm{Cay}$, which is $\theta_1$--hyperbolic for some $\theta_1$ because $(\widehat{G},P)$ is relatively hyperbolic. Since $\widehat{Y}$ is not a tree (it has loops deep in horoballs), $\theta_1 \geq 1$. Increasing $\theta_1$ if necessary, assume $\theta_1 \geq \theta_0$, where $\theta_0$ is the hyperbolicity constant of horoballs from \thref{horoballs are hyperbolic}. Let $\sH' \subset \widehat{Y}$ be the combinatorial horoball attached to $P$ viewed as a subset of $\textrm{Cay}$. Then $P = \textrm{Stab}_{G}(\sH')$ and because $P$ acts transitively on itself, $\leftQ{\sH'}{P}$ is quasi--isometric to a ray (it is literally a ray with loops attached coming from deep horizontal edges). Let $\sS$ be the points of $\sH'$ at depth exactly $500\theta_1$ and let $\mathbb{S} = \{g\sS \;| \; g \in \widehat{G}\}$. Note that each element of $\mathbb{S}$ is connected because each $P$ is connected as a subset of $\textrm{Cay}$. Let $\sH$ be the points of $\sH'$ of depth $\geq 500\theta_1$ and let $\widecheck{Y}$ be the points of $\widehat{Y}$ with depth $ \leq 500\theta_1$ so that 
    \[\widehat{Y} = \widecheck {Y} \cup_\mathbb{S}\bigg(\bigcup_{g \in \widehat{G}}g\sH_i\bigg).\]
    Note that $G$ acts freely and cocompactly on $\widecheck{Y}$. Indeed, $\widehat{G}$ acts transitively on itself, so the quotient is a path of length $500\theta_1$, possibly with some loops attached. Thus $(\widehat{Y},\widecheck{Y})$ is a pseudo--cusped space for $(\widehat{G},P)$.

    Further, the horoballs $\{g\sH\;| \; g \in \widehat{G}\}$ are $10^3\theta_1$--separated because any geodesic connecting distinct $g\sH, h\sH$ must travel up through $g\sH'$ into $\textrm{Cay}$, through $\textrm{Cay}$ to $hP$, and down through $h\sH'$ to $h\sH$. The first and last leg of this journey have length $500\theta_1$, for a total length of at least $10^3\theta_1$. 

    Let $\sigma_0= 10^7\theta_1$. By \thref{injective on balls centered on cayley graph}, there is a finite subset $F$ so that if $N \triangleleft P$ is a filling kernel with $N \cap F = \varnothing$ and $K = \llangle N \rrangle_{\widehat{G}} $, then any ball of radius $500\theta_1 + 6\sigma_0$ in $\widehat{Y}$ centered at a point in $\textrm{Cay}$ maps injectively into $\widehat{Y}/K$. It follows that for any such $K$, this quotient is also injective on balls of radius $6\sigma_0$ with center in $\widecheck{Y}$, since this ball is contained in a ball of radius $500\theta_1 + 6\sigma_0$ centered at a point in $\textrm{Cay}$. This is the $F$ promised in the statement.

    \begin{claim}\thlabel{only claim}
    Suppose $N\triangleleft P$ is a filling kernel with $N \cap F = \varnothing$.
    If $n \in N \setminus \{1\}$ and $x \notin \sH$, then $d_{\widehat{Y}}(x,n \cdot x) \geq 10^4\theta_1$. 
\end{claim}
\begin{proof}[Proof of Claim]
    Fix $n$ and let $K = \llangle N\rrangle_{\widehat{G}}$.
    
    First consider a point $y\in \sH$ with depth exactly $500\theta_1$ and let $y'$ be the point above $y$ with depth $0$. For a contradiction, suppose $d_{\widehat{Y}}(y,n\cdot y) < 10^5\theta_1$, then by the triangle inequality,
    \[d_{\widehat{Y}}(y',n\cdot y) \leq 500\theta_1 + 10^5\theta_1 < 10^7\theta_1 \leq \sigma_0.\] 
    Thus $y,n\cdot y\in B(y',\sigma_0)$. Since $n \in K$, this means the quotient $\widehat{Y}\longrightarrow \widehat{Y}/K$ is not injective on $B(y',\sigma_0)$, which contradicts the choice of $F$. This implies $d_{\widehat{Y}}(y,n\cdot y) \geq 10^5\theta_1$.

    Now suppose $x$ is any point of $\widehat{Y}\setminus \sH$, and choose a point $y \in \sH$ so that $d_{\widehat{Y}}(x,\sH) = d_{\widehat{Y}}(x,y)$. Note that $y$ must have depth $500\theta_1$, so the previous paragraph implies $d_{\widehat{Y}}(y,n\cdot y) \geq 10^5\theta_1$. By \cite[Lemma 3.10]{BoundariesOfDehnFillings}, we can choose a geodesic $\gamma$ from $y$ to $n\cdot y$ which has $2$ vertical segments and one horizontal segment of length at most $3$. Because $\gamma$ has length $d_{\widehat{Y}}(y,n\cdot y) \geq 10^5\theta_1$, all but $3$ of the edges of $\gamma$ are vertical, and $y,n \cdot y$ are at depth $500\theta_1$, we know $\gamma$ must go to a depth of at least $500\theta_1 + 10^5\theta_1 - 3$. Consider a geodesic quadrilateral with corners $x,y,n\cdot y, n\cdot x$. It is a standard fact that a geodesic $n$--gon in a $\delta$--hyperbolic space is $(n-2)\delta$--slim, hence this quadrilateral is $2\theta_1$--slim. The geodesics $[x,y]$ and $[n\cdot x, n\cdot y]$ do not go deeper than $500\theta_1$ in $\sH$ because of the choice of $y$. Therefore the points of maximal depth on $\gamma$ must be within $2\theta_1$ of some point on $[x,n \cdot x]$. Thus $[x,n \cdot x]$ must reach a depth of at least $498 \theta_1 + 10^5\theta_1 - 3$ in $\sH'$. Now $x,n\cdot x \notin \sH_i$, so $x,n\cdot x$ could have depth at most $500\theta_1 - 1$. Putting this together, we see that 
    \[d_{\widehat{Y}}(x,n\cdot x) \geq (498\theta_1 + 10^5\theta_1 - 3) - (500\theta_1 - 1) = (10^5 - 2)\theta_1- 2 \geq 10^4\theta_1,\]
    where the last step uses the fact that $\theta_1 \geq 1$. This completes the proof of \thref{only claim}.
\end{proof}

Fix a finite index subgroup $\widehat{H}$ of $\widehat{G}$ and an induced peripheral structure $(\widehat{H},\{P_i\}_i)$ with $g_i \in \widehat{G}$ and $P_i := \widehat{H} \cap P^{g_i}$. Note that $(\widehat{Y},\widecheck{Y})$ is clearly a pseudo--cusped space for $(\widehat{H},\{P_i\}_i)$, since \thref{defn pseudo cusped space} is a list of assumption on the graphs $\widehat{Y},\widecheck{Y}$, and the $\{P_i\}_i$ are exactly a collection of representatives for $\widehat{H}$--stabilizers of horoballs in $\widehat{Y}$. 

\begin{claim}
    Suppose $N \triangleleft P$ is a filling kernel and $(\widehat{H},\{P_i\}_i, \{N_i\}_i)$ is the induced group triple for $\widehat{H}$. Let $K = \llangle N \rrangle_{\widehat{G}}$ and $K_{\widehat{H}} = \llangle N_i\rrangle_{\widehat{H}}$. 
    \begin{enumerate}
        \item If $\widehat{H}$ is normal in $\widehat{G}$ and $\widehat{H} \cap F = \varnothing$, then the action of $(\widehat{H},\{P_i\}_i)$ and $K_{\widehat{H}}$ satisfies the assumptions of \thref{main metric lemma try2} and the very translating property.
        \item If $N \cap F = \varnothing$, then the action of $(\widehat{G},P)$ and $K$ satisfies the assumptions of \thref{main metric lemma try2} and the very translating property. 
    \end{enumerate}
\end{claim}
\begin{proof}[Proof of Claim]
    For the first assertion, suppose further that $\widehat{H}$ is normal and $\widehat{H} \cap F = \varnothing$. Then $\widehat{H} \cap N$ is normal in $P$, and by (1) of \thref{induced filling for normal subgroups}, the filling of $\widehat{H}$ by $N$ is the same as the filling of $\widehat{H}$ by $\widehat{H} \cap N$, so we consider filling of $\widehat{H}$ induced by $\widehat{H} \cap N$ instead and let $K' = \llangle \widehat{H} \cap N \rrangle_{\widehat{G}}$. By \thref{induced filling}, the filling kernels for $\widehat{H}$ induced by $\widehat{H} \cap N$ are $\widehat{H} \cap (\widehat{H} \cap N)^{g_i} = \widehat{H} \cap N^{g_i} = N_i$. Since $K_{\widehat{H}} = \llangle N_i \rrangle_{\widehat{H}}$ and each $N_i$ is a subgroup of a conjugate of $\widehat{H} \cap N$, we have $K_{\widehat{H}} < K'$, and the following diagram commutes.
    \begin{center}\[
\begin{tikzcd}
\widehat{Y} \arrow[rr] \arrow[rd] &                            & \leftQ{\widehat{Y}}{K'} \\
                                  & \leftQ{\widehat{Y}}{K_{\widehat{H}}} \arrow[ru] &              
\end{tikzcd}
\]\end{center}

    Because $(\widehat{H} \cap N) \cap F = \varnothing$, the choice of $F$ implies that balls of radius $6\sigma_0$ centered at a point in $\widehat{Y}$ embed into $\leftQ{\widehat{Y}}{K'}$. Since the diagram commutes, these balls also embed into $\leftQ{\widehat{Y}}{K_{\widehat{H}}}$, satisfying the injectivity radius requirement of \thref{main metric lemma try2}. For the very translating assumption, we noted above that the elements of $\{g\sH \; | \; g \in \widehat{G}\}$ are $10^3\theta_1$--separated. Further, \thref{only claim} and the fact that $(\widehat{H} \cap N) \cap F = \varnothing$ imply that the action of $K'$ on $(\widehat{Y},\widecheck{Y})$ satisfies the very translating condition \thref{geometrically CH}, hence so does the action of $K_{\widehat{H}}$ on $(\widehat{Y},\widecheck{Y})$. This completes the proof of the first assertion.

    For the second assertion, suppose $N \cap F = \varnothing$. Then the choice of $F$ implies the action of $(\widehat{G},P)$ and $\llangle N \rrangle_{\widehat{G}}$ satisfies the injectivity requirement of \thref{main metric lemma try2}, and again the very translating condition for the action of $K$ is implied by the separation of horoballs and \thref{only claim}. This proves the claim.
\end{proof}

    To see how this claim implies the desired result, suppose $N \triangleleft P$ is any filling kernel so that $(\widehat{G},P,N)$ is a drilling of $\widehat{G}/\llangle N \rrangle_{\widehat{G}}$ along the image of $P$. Let $Q$ be this image and let $H$ be the induced filling of $\widehat{H}$ by $N$. 
    
    We prove Theorem \ref{Main Theorem} (1). By the previous claim, if $\widehat{H} \cap F = \varnothing$, then the action of $(\widehat{H},\{P_i\}_i)$ satisfies the assumptions of \thref{main metric lemma try2}, so Theorem \ref{effective main theorem} (1) implies $H$ is a branched cover of $G$ along $Q$. If we assume further that $(\widehat{G},P)$ is a $PD(3)$ pair and $N \cong Q \cong \ZZ$, then the previous claim implies the appropriate very translating property, so that Theorem \ref{effective main theorem} implies $H$ is a $PD(3)$ group with $\partial H \cong S^2$. The proof of Theorem \ref{Main Theorem} (2) uses the previous claim and Theorem \ref{effective main theorem} in essentially the same way.
\end{proof}

Our metric and cohomological work has borne fruit. However, we believe that the extra assumptions that $(\widehat{G},P)$ is a $PD(3)$ pair and that $N \cong Q \cong \ZZ$ used to calculate the boundary of our branched covers are unnecessary. We believe that the right metric assumptions and arguments would show that all these branched covers have boundary $S^2$. A totally different approach to building branched covers would be to adapt the `spiderwebs' of \cite{BoundariesOfDehnFillings}. Given a relatively hyperbolic group $(G,\mathbb{P})$ with $\partial_\mathbb{P}G \cong 
S^2$ and a long enough filling kernel $K \triangleleft G$, the authors approximate $K$ by bigger and bigger subgroups to build a sequence of relatively hyperbolic groups converging to $\overline{G} = G/K$ in a suitable sense. The Bowditch boundaries of these groups are all $2$--spheres, and their boundaries weakly Gromov--Hausdorff converge to $\partial \overline{G}$. They then show that the limit of spheres must be a sphere, hence $\partial \overline{G} \cong S^2$. A similar strategy is used in \cite{Drilling} to calculate that $\partial \widehat{Y} \cong S^2$. One could hope that if $(\widehat{G},P,N)$ is a drilling of $G$ along $g$ with $\partial \overline{G} \cong S^2$ and $N$ is long enough to apply the strategy in \cite{BoundariesOfDehnFillings}, this sequence of approximations would induce a sequence of approximations for the induced filling of a subgroup, allowing us to calculate the boundary for any branched cover without assuming $(\widehat{G},P)$ is a $PD(3)$ pair. This is a strong assumption which implies $G$ is torsion free and more, and it would be nice to remove it.

\subsection{Drillings from \cite{Drilling} are deep enough}

Finally, we show that the drillings produced in \cite{Drilling} are deep enough so that every finite index subgroup produces a branched cover. Unfortunately the metric properties in this construction are not strong enough to imply the injectivity condition of \thref{main metric lemma try2}, so we cannot simply apply it directly. Instead, we imitate the proof of \thref{main metric lemma try2} in this special case. To begin, we restate \cite[Theorem B]{Drilling}. 

\begin{theorem}
    Let $X$ be a hyperbolic graph with boundary $S^2$. For every hyperbolic isometry $g$ of $X$ with ($g$--invariant, quasi--geodesic) axis $\ell$, there exists a $\Sigma > 0$ so that the following holds:

    Suppose $G$ is a subgroup of $Isom(X)$ so that 
    \begin{enumerate}
        \item $G$ acts freely and cocompactly on $X$; and
        \item For any $h \in G\setminus \langle g\rangle$, the axes $\ell$ and $h\cdot \ell$ are at least $\Sigma$ apart.
    \end{enumerate}
    Then there exists a drilling $(\widehat{G},P)$ of $G$ along $g$ which satisfies the following properties:
    \begin{enumerate}
        \item $(\widehat{G},P)$ is relatively hyperbolic with $P$ either free abelian of rank $2$ or the fundamental group of the Klein bottle;
        \item The Bowditch boundary of $(\widehat{G},P)$ is $S^2$.
    \end{enumerate}
    The peripheral subgroup $P$ is free abelian if and only if $G$ preserves an orientation on $\partial X$. Finally, $\widehat{G}$ is torsion free if and only if $G$ is torsion free.
\end{theorem}

\drillingsAreDeepEnough*

\begin{notation}\thlabel{properties of Yhat}
    For the rest of this subsection, fix a $(\widehat{G},P,N)$ constructed through the proof of \cite[Theorem B]{Drilling} and assume $G$ is torsion free. Let $K = \llangle N \rrangle_{\widehat{G}}$. We list the properties we need from this construction, and note that the constants are described much more explicitly in \cite[Notation 7.2, \S 8.1]{Drilling}.
    \begin{enumerate}
        \item $\widehat{G}$ acts freely and cocompactly on a metric graph $\widecheck{Y}$.

        \item $\widehat{Y}$ is constructed from $\widecheck{Y}$ by attaching combinatorial horoballs to a $\widehat{G}$ invariant collection of subgraphs $\mathbb{S}$, and $\widehat{Y}$ is $\delta_2 = 1500\delta_1$--hyperbolic \cite[Definition 8.3, Corollary 8.5]{Drilling}.

        \item $\widecheck{Y}$ is $D_1$--simply connected. In particular, the quotient $q:\widecheck{Y}\longrightarrow \leftQ{\widecheck{Y}}{K} =:\widecheck{X}_0$ is a $D_1$--universal covering map \cite[Lemma 8.6]{Drilling}, and $D_1 \geq 5$.
        
        \item Each homotopically nontrivial loop in $\widecheck{X}_0$ has length at least $\sys_0 \geq 2^{25\sigma_0}$, where $\sigma_0 = \max\{10^7\delta_1,10^5D_1\}$.

        \item Any ball of radius $\sigma_0$ in $\widehat{Y}$ is isometric to a ball in a $\delta_1$--hyperbolic space \cite[Proof of Corollary 8.5]{Drilling}.
        
        \item  $\widehat{G}$ acts transitively on the collection $\mathbb{S}$, and the subgroup $P$ is defined by fixing some $\widetilde{\sS} \in \mathbb{S}$ and setting $P:=\textrm{Stab}_{\widehat{G}}(\widetilde{\sS})$. $P$ is a semidirect product $\langle a\rangle \rtimes \langle b\rangle$, where $N:= \langle a \rangle$ and the image of $b$ in $G$ is $g$ \cite[Definition 8.8, Lemma 8.10]{Drilling}. 

        \item $(\widehat{G},P,N)$ is a very translating triple on $(\widehat{Y}, \widecheck{Y})$ \cite[proof of Theorem 11.3]{Drilling}.
        \item $\partial_P\widehat{G} \cong \partial \widehat{Y} \cong S^2$.
    \end{enumerate}
\end{notation}

\begin{notation}
    We fix a finite index subgroup $\widehat{H}$ of $\widehat{G}$ and let $(\widehat{H},\{P_i\}_i,\{N_i\}_i)$ be the induced group triple from $(\widehat{G},P,N)$. Let $\overline{P}_i = P_i/N_i$, $K_{\widehat{H}} = \llangle N_i\rrangle_{\widehat{H}}$ and $H = \widehat{H}/K_{\widehat{H}}$. Let $p:\widecheck{Y}\longrightarrow \leftQ{\widecheck{Y}}{K_{\widehat{H}}} =:\widecheck{\Gamma}$ be the natural quotient. For each $i$, let $\sS_i$ be the horosphere in $\mathbb{S}$ so that $P_i = \textrm{Stab}_{\widehat{H}}(\sS_i)$, and let $\overline{\sS}_i = p(\sS_i)$. Construct $\widehat{\Gamma}$ by attaching a combinatorial horoball to $p(\sS)$ for each $\sS \in \mathbb{S}$.  Let $\overline{\sH}_i$ be the horoball of $\widehat{\Gamma}$ attached to $\overline{\sS}_i$. 
\end{notation}

\begin{lemma}
    The set $\{\overline{\sH}_i\}_i$ is a collection of representatives for distinct $H$--orbits of horoballs in $\widehat{\Gamma}$. If $\overline{\sH}$ is a horoball of $\widehat{\Gamma}$, then $\textrm{Stab}_{H}(\overline{\sH})$ is a conjugate of some $\overline{P}_i$ which acts geometrically on $\overline{\sS}_i$.
\end{lemma}
\begin{proof}
    This is exactly the same as \thref{stabilizer of image of horoball is image of stabilizer} with different letters. 
\end{proof}

\begin{lemma}\thlabel{simply connected special case}
    $\widehat{\Gamma}$ is $D_1$--simply connected.
\end{lemma}
\begin{proof}
    Because $K_{\widehat{H}} < K$ and $\widecheck{Y}$ is $D_1$--simply connected, we have the following commutative diagram, where $p,q$ are $D_1$--universal covering maps with deck groups $K_{\widehat{H}}$ and $K$ respectively. 
    \begin{center}
    \begin{tikzcd}
    \widecheck{Y} \arrow[rr, "q"] \arrow[rd, "p"'] &                                    & \widecheck{X}_0 \\
                                               & \widecheck{\Gamma} \arrow[ru, "F"'] &            
    \end{tikzcd}
    \end{center}
    Standard covering space theory tells us $\pi_1^{D_1}(\widecheck{\Gamma}) = K_{\widehat{H}}$. Since $K_{\widehat{H}} = \llangle N_i \rrangle_{\widehat{H}}$ and each $N_i$ stabilizes the connected subgraph $\sS_i$, we see that $\pi_1^{D_1}(\widecheck{\Gamma})$ is generated by loops which are freely homotopic to loops contained in the horospheres of $\widecheck{\Gamma}$. Suppose $\overline{\sS}$ is such a horosphere with horoball $\overline{\sH}$. Then any loop in $\overline{\sS}$ becomes nullhomotopic in $\overline{\sH}^{D_1}$ because horoballs are $5$--simply connected by \thref{horoballs are simply connected} and $D_1 \geq 5$. Thus generators of $\pi_1^{D_1}(\widecheck{\Gamma})$ become nullhomotopic in $\widehat{\Gamma}^{D_1}$, and $\pi_1^{D_1}(\widehat{\Gamma})$ is trivial. 
\end{proof}

The next lemma is a very similar to \cite[Lemma 7.9]{Drilling}, and we write out the full argument here for clarity. The strategy is to use fact that nontrivial loops in $\widecheck{X}_0$ are very long to prove that the projection $\widecheck{Y}\longrightarrow \widecheck{X}_0$, hence $p$, has a large injectivity radius. This injectivity radius is not large enough to satisfy the third requirement of \thref{main metric lemma try2}, which is why we cannot apply \thref{main metric lemma try2} directly. To see this more explicitly, note that $\sigma_0 := \max(10^7\delta_1,10^5D_1)$ and $\delta_2 :=1500\delta_1$, but in order to satisfy the third requirement of \thref{main metric lemma try2}, we would need an injectivity radius of $\max(10^7\delta_2,10^5 D_1) \geq 10^7(1500\delta_1)$. Unless $D_1$ is very big compared to $\delta_1$, this is least $1500\sigma_1$. The argument below only provides an injectivity radius of $9\sigma_0$, and we cannot really hope to improve on this without adjusting the lower bound on $\sys_0$, or coming up with a different argument.  

\begin{lemma}
    For any $z \in \widehat{Y}$ which is either not in a horoball or lies at depth at most $3\sigma_0$ in a horoball, the projection $p:\widehat{Y}\longrightarrow \widehat{\Gamma}$ is an isometric embedding on $B(z,3\sigma_0)$ and
    \[p(B(z,3\sigma_0)) = B(p(z),3\sigma_0).\]
\end{lemma}
\begin{proof}
    We claim $p$ is injective on $B(z,9\sigma_0)$ so that \thref{upgrade to isometry} implies $p$ is an isometric embedding on $B(z,3\sigma_0)$. For a contradiction, suppose $z_1,z_2 \in B(z,9\sigma_0)$ are distinct and $p(z_1) = p(z_2)$. By the triangle inequality, $d_{\widehat{Y}}(z_1,z_2) \leq 18\sigma_0$. Since $z$ has depth at most $3\sigma_0$, the $z_i$ have depth at most $12\sigma_0$, with this bound only being achieved by traveling vertically straight down into the horoball. Because the action on $\widehat{Y}$ is depth preserving, the $z_i$ must have the same depth and the entire vertical paths above and below the $z_i$ are identified in $\widehat{\Gamma}$. If the depth of the $z_i$ is positive, we replace them with points directly above them on the horosphere so that they have depth $0$. Thus at the cost of assuming 
    \[d_{\widehat{Y}}(z_1,z_2) \leq (12+18+12)\sigma_0 = 42\sigma_0,\]
    we may suppose that $z_1,z_2$ have depth $0$. Because they are at depth $0$, we can view the $z_i$ as points in $\widecheck{Y}$. Recall the following commutative diagram, where $p,q$ are $D_1$--universal covering maps with deck groups $K_{\widehat{H}}$ and $K$ respectively. 

    \begin{center}
    \begin{tikzcd}
    \widecheck{Y} \arrow[rr, "q"] \arrow[rd, "p"'] &                                    & \widecheck{X}_0 \\
                                               & \widecheck{\Gamma} \arrow[ru, "F"'] &            
    \end{tikzcd}
    \end{center}
    Because $p(z_1) = p(z_2)$ and the diagram commutes, we must have $q(z_1) = q(z_2)$. Since $q$ is a $D_1$--covering map, any geodesic between $z_1$ and $z_2$ maps to a loop representing a nontrivial element of $\pi_1^{D_1}(\widecheck{X}_0)$. Such loops have length at least $\sys_0$, hence $d_{\widecheck{Y}}(z_1,z_2) \geq \sys_0 \geq 2^{25\sigma_0}$. To compare this distance to $d_{\widehat{Y}}(z_1,z_2)$, notice that because $\widehat{Y}$ is constructed by attaching horoballs to $\widecheck{Y}$, any segment which goes through a horoball may be pushed out of the horoball at the cost of increasing the length by an exponential function. This cost is maximized when the geodesic is fully contained in horoball. More precisely, using the bound in \thref{horoballs are hyperbolic} and the fact that paths which are pushed out of a horoball have length at least $4$, we have 
    \[d_{\widecheck{Y}}(z_1,z_2) \leq 2\cdot 2^{\frac{1}{2}d_{\widehat{Y}}(z_1,z_2)} \leq 2^{21\sigma_0+1} < 2^{25\sigma_0} \leq \sys_0.\]
    This contradicts the lower bound $d_{\widecheck{Y}}(z_1,z_2) \geq \sys_0$.
\end{proof}

\begin{lemma}\thlabel{modelling for special case}
    Any ball of radius $\sigma_0$ in $\widehat{\Gamma}$ is isometric to a ball in a $\delta_1$ hyperbolic space. 
\end{lemma}
\begin{proof}
    By \thref{upgrade to isometry}, it suffices to show any ball of radius $3\sigma_0$ in $\widehat{\Gamma}$ maps homeomorphically to a ball in some $\delta_2$--hyperbolic space. Fix a point $z \in \widehat{Y}$, and let $B^+ = B(z,3\sigma_0)$. As in the proof of \thref{modelling for criterion}, there are two cases, depending on how deep $z$ is in a horoball. If $z$ has depth at least $3\sigma_0$ in a horoball $\overline{\sH} \subset \widehat{\Gamma}$, then the inclusion $B^+ \longrightarrow \overline{\sH}$ is a homeomorphism and all combinatorial horoballs are $\theta_0$--hyperbolic. Since $\theta_0 \leq \delta_1$, this is more than enough. If $z$ has depth less than $3\sigma_0$, then the previous lemma implies $B^+$ is isometric to a ball in $\widehat{Y}$. By \thref{properties of Yhat}, each ball of radius $\sigma_0$ in $\widehat{Y}$ is isometric to a ball in $\delta_1$--hyperbolic space.
\end{proof}

\begin{lemma}
    The space $\widehat{\Gamma}$ is $\delta'= 300\delta_1$--hyperbolic.
\end{lemma}
\begin{proof}
    This follows from the Coarse Cartan Hadamard \thref{coarseCH} with balls of radius $\sigma_0$ and Lemmas \ref{modelling for special case}, \ref{simply connected special case}.
\end{proof}

\begin{lemma}
    $(H,\{\overline{P}_i\}_i)$ is relatively hyperbolic.
\end{lemma}
\begin{proof}
    Just as in the concluding step of \thref{main metric lemma try2}, $(\widehat{\Gamma},\widecheck{\Gamma})$ is a pseudo--cusped space for $(H,\{\overline{P}_i\})$, hence $(H,\{\overline{P}_i\})$ is relatively hyperbolic by \thref{pseudo cusped implies rel hyp}.  
\end{proof}

\begin{proof}[Proof of Theorem \ref{drillings are deep enough}]
    By the previous lemma, $(H,\{\overline{P}_i\}_i)$ is relatively hyperbolic, and by \thref{induced filling for normal subgroups} (3), each $\overline{P}_i$ is isomorphic to $\ZZ$. Thus $H$ is hyperbolic because it is hyperbolic relative to hyperbolic subgroups. 

    If in addition $G$ is torsion free, then $\widehat{G}$ and hence $\widehat{H}$ is as well (as noted in \cite[Corollary D]{Drilling}, this assumption implies $(\widehat{G},P)$ is a $PD(3)$ pair).
    Since $\partial_P\widehat{G} \cong S^2$, \thref{cohom boundary connector} implies $(\widehat{H},\{P_i\}_i)$ is a $PD(3)$ pair. By \thref{properties of Yhat} (7), $(\widehat{G},P,N)$ is a very translating triple, so the induced group triple for $(\widehat{H}$ is Cohen--Lyndon, see \thref{finite index subgroup of geom CH}. As in previous proofs, \thref{filled group is PDn} applies to the filling of $(\widehat{H},\{P_i\}_i)$ by $\{N_i\}_i$ and implies $H$ is a $PD(3)$ group and $\partial H \cong S^2$ by \thref{bestvina boundary calculator}. 
\end{proof}